%% file: timePer2Pt.tex
\documentclass[11pt,reqno,oneside,a4paper]{article}
\usepackage[a4paper,includeheadfoot,left=25mm,right=25mm,top=00mm,bottom=20mm,headheight=20mm]{geometry}
\input{texhead-main} % Standard packages, page layout, theorem environments, macros, etc
\input{texhead-project} % Macros specific to this project.
\author{
    A.\ S.\ Fokas\textsuperscript{*}, B.\ Pelloni\textsuperscript{**} and D.\ A.\ Smith\textsuperscript{\textdagger} \\
    \footnotesize\textsuperscript{*} DAMTP, University of Cambridge, UK, \href{mailto:t.fokas@cam.ac.uk}{\texttt{t.fokas@cam.ac.uk}} \\
    \footnotesize\textsuperscript{**} Heriot-Watt University \& Maxwell Institute for the Mathematical Sciences, Edinburgh, UK, \\ \footnotesize\href{mailto:b.pelloni@hw.ac.uk}{\texttt{b.pelloni@hw.ac.uk}} \\
    \footnotesize\textsuperscript{\textdagger} Yale-NUS College \& Department of Mathematics, National University of Singapore, Singapore, \\
    \footnotesize\href{mailto:dave.smith@yale-nus.edu.sg}{\texttt{dave.smith@yale-nus.edu.sg}}
}
\newcommand\theauthorshort{ASF, BP \& DAS}
\title{Time-periodic linear boundary value problems on a finite interval}
\renewcommand{\runningtitle}{Time-periodic 2-point BVPs}
\date{\today}

\begin{document}
\maketitle
\thispagestyle{fancy}

\begin{abstract}
We study the large time behaviour of the solution of a linear dispersive PDEs posed on a finite interval, when the prescribed boundary conditions are time periodic. We use the approach pioneered in \cite{FL2012b} for nonlinear integrable PDEs. and then applied to linear problems on the half-line in \cite{fvdw2021}, to characterise necessary conditions for the solution of such a problem to be periodic, at least in an asymptotic sense. We then fully describe the periodicity properties of the solution in three important illustrative examples, recovering known results for the second-order cases and establishing new ones for the third order one.
\end{abstract}

\section{Introduction} \label{sec:intro}
In this paper, we study the large time behaviour of the solution of a linear dispersive PDE posed on a finite interval, when the prescribed boundary conditions are time periodic. From the mathematical point of view this is a classical problem. From the phenomenological point of view, such problems arise as linearised models for experimental situations in which a periodic boundary input is driving, solely or in conjuction with a given initial state,  the dynamics of the model, e.g. waves along a finite channel, or temperature in a finite rod.

To illustrate our general method and the possible scenarios that can unfold, we present in detail the analysis of the following three specific PDEs of mathematical physics:
\begin{eqnarray}
  \label{lS}
   & \mbox{(linear Schr\"odinger)}&iu_t+u_{xx}=0,\\
   \label{heat}
     & \mbox{(heat)} & u_t-u_{xx}=0,\\
    \label{lkdv}
  & \mbox{(Stokes)}& u_t+u_{xxx}=0.
    \end{eqnarray}
We will study the solution of the problem posed for $x\in(0,1)$, $t>0$, with prescribed initial condition at $t=0$ and appropriate time-periodic boundary conditions, with a common period $T$, at $x=0$ and $x=1$.

We follow the general idea  proposed in \cite{fvdw2021} of analysing the problem in the complex spectral space to characterise necessary  conditions for the solution to be time-periodic, either exactly or in an asymptotic sense.
This idea arises naturally in the  context of the {\em unified transform} (also known in the literature as the {\em Fokas transform}), introduced by one of the authors to solve linear and integrable nonlinear boundary value problems \cite{Fok1997a, Fok2008a}; see~\cite{DTV2014a} for a pedagogical introduction.
For nonlinear problems, this approach leads to the definition of the {\em $Q$ equation}, and the solution rests on the analysis of this equation.
The $Q$ equation can be obtained as one of the relations in the so-called Lax pair formulation of the given integrable PDE \cite{fvdw2021, FL2012b}, which is the starting point for deriving the unified transform representation in its general form, namely in a form valid for both  linear and integrable nonlinear PDEs.
For linear problems, the analogous equation is obtained by an application of the Fourier transform.
Indeed, this is the approach we use in the present paper.
However, we stress that the more general derivation from the Lax pair provides a natural connection between the ideas that underpin the formulation and development of the powerful unified transform method, and the starting point of our study of time-periodic problems.
The crucial insight is that, following the strategy of the unified transform approach, we work in the complex (rather than the real) spectral plane, and make use of the constraints imposed by analyticity requirements in this complex space.

\smallskip
Our results in this paper imply that
\begin{enumerate}
    \item
    our complex spectral method can be used to study the  large time asymptotics of the so-called Dirichlet-to-Neumann map, and hence of the solution of the given initial boundary value problem;
     \item
    our method yields necessary conditions for asymptotic periodicity, via an analyticity argument, in a more natural and direct way when compared to  earlier partial results (e.g. see \cite{Duj2009a});
     \item
     some of the earlier methods, including the approach taken by Dujardin in \cite{Duj2009a}, do not extend, as one might have expected, to other dispersive problems of higher order, or to more general boundary conditions. Indeed, they depend on a special choice of initial and boundary conditions and on the existence of an explicit Fourier series representation of the solution, which does not hold in general. In contrast, our method appears to have more general applicability. Indeed,  we conjecture that it can be used for general dispersive linear PDEs on a finite interval.
\end{enumerate}

\smallskip
Time-periodic linear boundary value problems posed on a half-line were studied recently in \cite{fvdw2021}, where it was shown that the large time behaviour of the solution, with given periodic boundary conditions, is always periodic.  In contrast, for problems posed on a finite interval we find that periodicity results hold only under certain necessary conditions, that  depend on the arithmetic relation between the two parameters of the problem: the period $T$ of the given boundary conditions, and the length $L$  of the finite interval. Some special boundary value problems for  equation  \eqref{lS} were studied in \cite{Duj2009a},  using classical asymptotic analysis of the associated  Fourier series.
In this paper, we use the more general unified transform representation of the solution and  complex analytical tools. This allows us to derive such necessary conditions in general for all three examples under consideration.
In order to keep the notational burden to a minimum  we always assume, without loss of generality,  that $L=1$.

The conditions imposed by our complex spectral approach, in addition to allowing us to derive necessary conditions for the existence of a time-periodic solution,  will be used to construct the first of the two building blocks needed in our general solution strategy. Namely, given specific time periodic boundary conditions with common period $T$, and assuming the necessary conditions imposed by analyticity hold,  we construct an explicit initial condition that guarantees that the solution $u_1(x,t)$ of the corresponding problem is {\em exactly} $T$-periodic.

The second building block in our strategy is the analysis of the solution $u_2(x,t)$  of a specific homogeneous boundary value problem, with no constraint on the initial condition. Depending on the PDE and type of boundary conditions, this solution will either (a) decay  in time; (b) be oscillatory but not periodic; or (c) have a time periodicity with period depending solely on the length of the finite space interval. The proof of the properties of this second solution is obtained through the analysis of the explicit solution representation. It is important to note that while standard or generalised  Fourier series (with respect to the space variable $x$)  can be used to represent the solution of second order problems, this is not possible in general for third and higher order cases. In the latter case, we can still use the more general and  powerful integral  representation given by the unified transform approach, see \cite{Pel2005a}.

In summary, in the cases when the necessary conditions for periodicity hold, we will write the solution of the given problem with time-periodic boundary conditions as
 $$
u(x,t)=u_1(x,t)+u_2(x,t),
 $$
 where
 \begin{itemize}
 \item
 $u_1(x,t)$ is the solution  corresponding to a certain explicit initial condition
 $$
 u(x,0)=u_T(x),$$ and the given $T$-periodic boundary conditions. The choice of $u_T$  guarantees that
$u_1(x,t)$ is {\em time periodic of period $T$}
  \item
$u_2(x,t)$ is the solution corresponding to initial condition
  $$
  u(x,0)=u_0(x)-u_T(x),$$
 and the homogeneous version of the prescribed boundary conditions.
\end{itemize}

\smallskip
The specific relation of these two special solutions drives the behaviour of time-periodic boundary value problems of the type we consider. We illustrate the possible behaviours
that can arise through the rigorous analysis of the three explicit examples of linear PDEs~\eqref{lS}--\eqref{lkdv}.

In particular, the arithmetic constraints between the length of the space interval and the time period appear when the solution $u_2(x,t)$ of the problem with homogeneous boundary conditions is itself time-periodic, rather than being decaying or oscillatory in time. This happens, for example, in the case of the linear Schr\"odinger equation with Dirichlet boundary conditions; in this case, the fundamental solution
$
\re^{ikx-ik^2t}$ evaluated at the eigenvalues $k=m\pi$ of the spatial problem is clearly time periodic with period $2/\pi$.
In contrast, for the heat equation the same argument yields functions  $
\re^{im\pi x-m^2\pi^2t}$ that decay as $t$ grows.

For the case of the third order Stokes equation, in general the spatial operator is not self-adjoint
and the location of its eigenvalues depends sensitively on the specific boundary conditions. Since the associated fundamental solution is  $
\re^{ikx+ik^3t}$, one may expect that the time-dependence is similar to the case of the linear Schr\"odinger, namely that the time dependence is characterised through a purely imaginary exponential. However the eigenvalues are only real for very specific boundary conditions, hence the exponent $ik^3t$ evaluated at the eigenvalues is not in general purely imaginary.

These structural differences are at the heart of the different results we obtain for the three examples we consider:

\begin{itemize}
    \item
   For equation~\eqref{lS},  we study the problem obtained when  time-periodic Dirichlet boundary conditions of period $T$ are given. We show that, generically, the solution is  time periodic only if $T$ and $2/\pi$ are linearly dependent over $\mathbb Q$, see Proposition~\ref{prop:LSPeriodic}.
    \item
   For equation~\eqref{heat}, we study the problem obtained when time-periodic Neumann  conditions of period $T$ are given. We show that the solution is always asymptotically time periodic, with the same period as the data, see Proposition~\ref{prop:Heat.NeumanntoDirichlet}.
    \item
    For equation~\eqref{lkdv}, we consider two types of time-periodic boundary data, either given independently at the two ends of the space interval, or given in a way that implies that the boundary conditions are coupled at the two endpoints.  We show that generically, in both cases,  the solution is asymptotically time periodic, with the same period as the data. These are the results of Propositions~\ref{prop:LKdV.Decoupled} and~\ref{prop:LKdV.Coupled} respectively. We also discuss  a particular but important case of coupled boundary conditions for which the associated spatial operator has real eigenvalues. In this case certain additional necessary conditions, given in Proposition~\ref{prop:LKdV.CoupledBeta1}, need to be satisfied.
\end{itemize}

In the present paper we consider only problems depending on one period $T$, common to the given data. We expect that it will be possible to generalise the analysis to the case that the given data have as many  different periods $T_i$ as there are given boundary conditions,   and study when it is possible to represent the solution in  additive parts of periods $T_i$. We leave this generalisation to future work.

\subsection{Notation and set-up of the problem} \label{sec:notation}

We pose all problems considered in this paper on the interval $[0,1]$, and denote the Fourier transform formally defined by
\[
   \widehat{\phi}(k) = \int_0^1 \re^{-\ri k x} \phi(x) \D x, \qquad
   \quad k \in\CC,
\]
where $\phi(x)$ is any function for which the integral is well defined.
While the Fourier transform  $\R$ is usually defined for $k\in\R$, for what follows it is crucial that we  allow $k\in\C$.

\smallskip
 Let $\Omega$  be a polynomial of degree $n$, and assume that $u(x,t)$ is a (smooth) solution of the PDE
\be\label{PDE}
    u_t + \Omega(-\ri\partial_x)u = 0,\qquad x\in(0,1).
\ee
We denote the boundary values of $u(x,t)$ by
\begin{align}\label{bvalues}
 g_j(t) &= \partial_x^j u(0,t), \nonumber \\
    h_j(t) &= \partial_x^j u(1,t).
\end{align}
Define formally
\begin{equation} \label{eqn:Q}
Q(k,t) = - \widehat{u}(k;t),\qquad k\in\C.
\end{equation}
By its definition, the function $Q$ is an entire function of $k$ for all $t\geq0$.

Let the polynomials $\{c_j(k)\}_0^{n-1}$ be defined by \cite{FP2005a}
\be\label{defn.cj}
   \sum_{j=0}^{n-1} c_j(k)\partial_x^j = \ri \left.\frac{\Omega(k)-\Omega(\ell)}{k-\ell}\right\rvert_{\ell=-\ri\partial_x}.
\ee
Taking the Fourier transform of \eqref{PDE},  we find the following equation for $Q$:
\be\label{Qeqn}
    Q_t(k,t) + \Omega(k) Q(k,t) = \sum_{j=0}^{n-1} c_j(k) (g_j(t)-\re^{-ik}h_j(t)), \quad k\in\C, \,t>0.
\ee
This is the linear analogue of the nonlinear {\em $Q$ equation} of \cite{FL2012b}.
Integration of this equation yields the {\em global relation}, that is of central importance to the Fokas transform method.

In this paper, we will consider equation \eqref{Qeqn} in the complex $k$-plane.

\subsection{Illustrative examples}
In our definition \eqref{PDE} of a linear dispersive PDE, $\Omega$ may be any polynomial of degree $N$, but we shall concentrate on monomials $\Omega(k)=ak^N$.
Ensuring that the fundamental solution $\re^{ikx-\Omega(k)t}$ of such a PDE is bounded for all $t>0$ requires
\be\label{acond}
\arg(a)\in\left[-\frac{\pi}{2},\frac{\pi}{2}\right].
\ee
It also places some restrictions on which linear combinations of boundary values appear in the boundary conditions~\cite{Pel2005a,Smi2012a,Smi2012b}, but we shall not discuss those restrictions here, and will consider only boundary conditions that satisfy this property.

\subsubsection*{Linear \texorpdfstring{Schr\"{o}dinger}{Schrodinger} equation} \label{sssec:LS}
As an example of dispersive, linear second-order PDE, we consider the time-dependent free space linear Schr\"{o}dinger equation \eqref{lS},
which is equivalent to the PDE~\eqref{PDE} with $\Omega(k)=\ri k^2$.
From equation~\eqref{defn.cj}, we find
\[
    c_0(k) = - k, \qquad\qquad c_1(k) = \ri,
\]
so the spectral equation~\eqref{Qeqn} is
\be\label{Qeqnls}
    Q_t(k,t) + \ri k^2 Q(k,t) = \ri \left[ \left(\ri k g_0(t) + g_1(t)\right) - \re^{-\ri k}\left(\ri k h_0(t) + h_1(t)\right) \right].
\ee

\subsubsection*{Heat equation} \label{sssec:heat}

The heat equation \eqref{heat}, arguably the most important second order linear PDE,
is equivalent to the PDE~\eqref{PDE} with $\Omega(k) = k^2$, so
\[
    c_0(k) = \ri k, \qquad\qquad c_1(k) = 1.
\]
The spectral equation~\eqref{Qeqn} is
\be\label{Qeqnheat}
    Q_t + k^2 Q = \left(\ri k g_0(t) + g_1(t)\right) - \re^{-\ri k}\left(\ri k h_0(t) + h_1(t)\right).
\ee

\subsubsection*{Stokes equation} \label{sssec:Lkdv}

The canonical example of linear dispersive third order PDE is the  Stokes equation \eqref{lkdv},
which is equivalent to the PDE~\eqref{PDE} with $\Omega(k) = -\ri k^3$, so
\[
    c_0(k) = k^2, \qquad\qquad c_1(k) = \ri k, \qquad\qquad c_2(k) = -1.
\]
Therefore the spectral equation~\eqref{Qeqn} is
\be\label{Qeqnstokes}
    Q_t -\ri k^3 Q = \left( k^2 g_0(t) - \ri k g_1(t) - g_2(t) \right) - \re^{-\ri k} \left( k^2 h_0(t) - \ri k h_1(t) - h_2(t) \right).
\ee

\subsection{Time-periodic boundary value problems on finite intervals}

The main question considered in this paper is whether the solution $u(x,t)$ is time-periodic, either in an exact or in an asymptotic sense, under the assumption that the prescribed boundary conditions are time-periodic.

In general, it turns out that it is too restrictive to require that the solution of such a boundary value problem be periodic, with the same period as the given boundary conditions, for all $t>0$. However, it is of interest to determine whether the time-dependence of the solution becomes periodic, with some period,  for large times.

A (sufficiently regular) function $d:\R^+\to \C$  is $T$-periodic, for  $T>0$, if and only if it can be represented via its exponential Fourier series
\be\label{dt}
    d(t) = \sum_{n\in\ZZ} D_n \re^{\ri n\omega t}, \qquad
\omega=\frac {2\pi}T.
\ee
We are interested in functions which behave approximately like periodic functions for large value of the time $t$.

\begin{defn} \label{defn:stronglyasymptoticallyperiodic}
    A function $d_a(t)$ is  \emph{asymptotically $T$-periodic} if there exists a function $d(t)$ which is  $T$-periodic and such that $\lim_{t\to\infty}\abs{d_a(t)-d(t)}=0$.
   It is \emph{strongly asymptotically $T$-periodic} if $d_a$ and $d'_a$ are both asymptotically periodic in such a way that there exists a continuously differentiable $T$-periodic function $d$ such that both $\lim_{t\to\infty}\abs{d_a(t)-d(t)}=0$ and $\lim_{t\to\infty}\abs{d'_a(t)-d'(t)}=0$ almost everywhere.
\end{defn}

Suppose that the solution $u(x,t)$ of a given boundary value problem on $[0,1]$  is strongly asymptotically $T$-periodic.
Then the boundary values $g_j$ and $h_j$ defined in \eqref{bvalues} are asymptotically periodic, and the function $Q$ is strongly asymptotically time periodic.
Therefore, there exist sequences $(G^{(j)}_n)_{n\in\ZZ}$ and $(H^{(j)}_n)_{n\in\ZZ}$ of complex numbers, and sequence $(q_n(k))_{n\in\ZZ}$ of complex functions for which, as $t\to\infty$,
\begin{equation} \label{eqn:AsymptoticBdryCoeffs}
\begin{aligned}
    g_j(t) &= \sum_{n\in\ZZ} G^{(j)}_n \re^{\ri n \omega t} + o(1), & Q(k,t) &= \sum_{n\in\ZZ} q_n(k) \re^{\ri n \omega t} + o(1), \\
    h_j(t) &= \sum_{n\in\ZZ} H^{(j)}_n \re^{\ri n \omega t} + o(1), & Q_t(k,t) &= \ri\omega \sum_{n\in\ZZ} n q_n(k) \re^{\ri n \omega t} + o(1),\;\; t\to \infty,
\end{aligned}
\end{equation}
where $\omega$ is defined in \eqref{dt}.

Substituting the above expression into the spectral equation~\eqref{eqn:Q}, keeping only the terms $\bigoh{1}$ in $t$, and using orthogonality of the exponential Fourier basis, we find
\begin{equation} \label{eqn:qn}
    q_n(k) = \frac{\sum_{j=0}^{N-1}c_j(k)\left( G^{(j)}_n - \re^{-\ri k}H^{(j)}_n \right)}{\ri n\omega + \Omega(k)}, \qquad n\in\ZZ.
\end{equation}

Because $u$ is strongly asymptotically time periodic, it must hold that
\[
    q_n(k) = - \frac{1}{T} \int_0^T \re^{-\ri n \omega t} \int_0^1 \re^{-\ri k x} u(x,t) \D x \D t.
\]
Therefore $q_n(k)$, just like $Q(k)$ itself, is an entire function of $k\in\C$.
The right side of equation~\eqref{eqn:qn} is expressed as a ratio of entire functions but, by the above argument, the ratio is itself entire.
Hence the following condition  must be satisfied:

\smallskip
\noindent{\em at each zero of the denominator on the right of equation~\eqref{eqn:qn},  the numerator must have a zero of the same order}.
\subsection{Characterisation of the Dirichlet-to-Neumann map}
Suppose $\Omega(k)=ak^N$, with $a$ satisfying \eqref{acond},   and let $\alpha=\exp(2\pi\ri/N)$ be a primitive $N$th root of unity.
Then, for $n\neq0$, the zeros of the denominator in \eqref{eqn:qn} are all simple, so
\begin{subequations} \label{eqn:qnnumerator}
\begin{equation} \label{eqn:qnnumerator.n}
    \forall\,n\in\ZZ\setminus\{0\},\;\forall\,r\in\{0,1,\ldots,N-1\}, \quad \left.\sum_{j=0}^{N-1}c_j(k)\left( G^{(j)}_n - \re^{-\ri k}H^{(j)}_n \right)\right\rvert_{k=\alpha^r\sqrt[n]{\ri n\omega/a}}=0.
\end{equation}
At $n=0$ there is a zero of order $N$, so
\begin{equation} \label{eqn:qnnumerator.0}
    \forall\,r\in\{0,1,\ldots,N-1\}, \quad \left.\frac{\D^r}{\D k^r} \sum_{j=0}^{N-1}c_j(k)\left( G^{(j)}_0 - \re^{-\ri k}H^{(j)}_0 \right)\right\rvert_{k=0}=0.
\end{equation}
\end{subequations}
For each $n\in\ZZ$, the corresponding equation of~\eqref{eqn:qnnumerator} is a rank $N$ system of $N$ equations in the $2N$ boundary value coefficients $G^{(j)}_n$ and $H^{(j)}_n$. Therefore we have proved the following general condition, which characterises the Dirichlet-to-Neumann map for (asymptotically) time periodic problems.

\begin{prop}[Dirichlet-to-Neumann map]
Consider the PDE \eqref{PDE}, with $\Omega(k)=ak^N$, with $a$ satisfying \eqref{acond}.
Assume that the PDE admits a unique solution $u(x,t)$ which is strongly asymptotically time periodic, with period $T$.
Let $g_j(t)$, $h_j(t)$ be its boundary values, defined in \eqref{bvalues}, which are then
(asymptotically) time periodic, with  $T$ the common period.

The coefficients $G^{(j)}_n$, $H^{(j)}_n$, $j=0,\ldots,N-1$, $n\in\Z\setminus\{0\}$, defined by \eqref{eqn:AsymptoticBdryCoeffs},  must satisfy the system \eqref{eqn:qnnumerator.n} and $G^{(j)}_0$, $H^{(j)}_0$ must satisfy \eqref{eqn:qnnumerator.0}.

In particular, assume that $N$ time periodic boundary conditions with period $T$ are prescribed. Then \eqref{eqn:qnnumerator.n}-\eqref{eqn:qnnumerator.0} is a system of $N$ equations for the remaining $N$ boundary values.  If this system admits a unique solution, this solution characterises the Dirichlet-to-Neumann map for the boundary value problem.
\end{prop}

In the rest of this paper, we specialise to the three examples, with specific time-periodic boundary conditions, and determine for each example the conditions that ensure that the solution is (asymptotically) periodic.

\section{The linear \texorpdfstring{Schr\"{o}dinger}{Schrodinger} equation} \label{sec:LS}
In this section we study the PDE~\eqref{lS}, with a given initial condition and time-periodic Dirichlet boundary conditions.

If we assume  that $u$ is strongly asymptotically time-periodic with period $T=2\pi/\omega$, equation~\eqref{eqn:qn} reduces to
\be\label{qnls}
    q_n(k) = \frac{\left(\ri k G^{(0)}_n+G^{(1)}_n\right) - \re^{-\ri k}\left(\ri k H^{(0)}_n+H^{(1)}_n\right) }{n\omega+k^2}, \qquad n\in\ZZ,
\ee
where the asymptotic Fourier coefficients are given by equations~\eqref{eqn:AsymptoticBdryCoeffs}.
For each $n\in\NN$, $q_{\pm n}$ must be entire, and the system~\eqref{eqn:qnnumerator.n} yields
\begin{subequations} \label{eqn:LS.qnnumerator}
\begin{align}
    \label{eqn:LS.qnnumerator.+n}
    0 &= \left(\ri k G^{(0)}_n+G^{(1)}_n\right) - \re^{-\ri k}\left(\ri k H^{(0)}_n+H^{(1)}_n\right), & k &= \pm\ri\sqrt{n\omega}, \\
    \label{eqn:LS.qnnumerator.-n}
    0 &= \left(\ri k G^{(0)}_{-n}+G^{(1)}_{-n}\right) - \re^{-\ri k}\left(\ri k H^{(0)}_{-n}+H^{(1)}_{-n}\right), & k &= \pm\sqrt{n\omega}.
\end{align}
Because $q_0$ is also entire, system~\eqref{eqn:qnnumerator.0} also holds, and reduces to the equations
\begin{equation} \label{eqn:LS.qnnumerator.n0}
    G_0^{(1)}=H_0^{(1)}, \qquad\qquad H_0^{(0)}-G_0^{(0)} = H_0^{(1)}.
\end{equation}
\end{subequations}
The time periodicity of 2-point boundary value problems for this equation was studied in  \cite{Duj2009a}. We  will reproduce Dujardin's results, concerned with necessary conditions for periodicity, by analysing the consequences of assuming that equation~\eqref{qnls} holds, i.e. by assuming the unique solvability of system \eqref{eqn:LS.qnnumerator.+n}. We will also give a simpler proof of the full classification of the possible behaviour of the solution.

\subsection{Existing results for linear \texorpdfstring{Schr\"{o}dinger}{Schrodinger} equation}

In his 2008 paper \cite{Duj2009a}, Dujardin proved a series of results for the 2-point Dirichlet boundary value problem for the linear Schr\"odinger equation~\eqref{lS}, with a vanishing  initial condition
\be\label{ls.ic}
    u(x,0)=0 \qquad\qquad x\in[0,1],
\ee
and time-periodic Dirichlet boundary conditions, with period $T$ and sufficient regularity.
As before, we denote
\be\label{ls.bc}
    u(0,t)=g_0(t), \qquad\qquad u(1,t)=h_0(t).
\ee

Below, we state Dujardin's results, specialised to the spatial interval of length $L=1$ on which we work.

\begin{theorem} \label{djt1}
    Assume that the period $T$ of the given Dirichlet boundary data $g_0(t)$ and $h_0(t)$ satisfies the relation
    \be\label{TLrel}
        T = \frac 2 \pi.
    \ee
    Then the solution of the boundary value problem~\eqref{lS},~\eqref{ls.ic}--\eqref{ls.bc} satisfies
    \be
        \lVert u(\cdot,t)\rVert_{\Lebesgue^\infty(0,1)}\rightarrow_{t\to\infty}\infty.
    \ee
 In particular, the solution cannot be  (asymptotically) periodic.
\end{theorem}

This is a special case of a more general result.
Given the linearity of the problem, in \cite{Duj2009a} the result is stated, without loss of generality,  for the problem with given $T$-periodic boundary conditions
\be\label{bcg0}
    u(0,t)=g_0(t)\mbox{ such that } g_0 \mbox{ is periodic,} \qquad u(1,t)=h_0(t):=0.
\ee
\begin{rmk}
    A slight loss of generality arises from the possibility of the long time effects of $g_0$ and $h_0$ interfering destructively.
    Dujardin remarks on this possibility, but does not investigate in detail.
\end{rmk}

\begin{theorem} \label{djt2}
    Assume that the period $T$ of the given Dirichlet boundary conditions and the factor $2/\pi$ are linearly dependent on $\QQ$, so that there exist $\alpha,\beta\in\Z\setminus\{0\}$, relatively prime, with $\alpha\geq 1$ and $\beta \leq -1$, such that
    \[
        \alpha\frac {2\pi}T+\beta\pi^2=0.
    \]
    Let
    \[
        {\cal R}=\{(n,m)\in\Z\times\Z: \,\alpha m^2-\beta n=0\}.
    \]
    If there exists  $(n,m)\in{\cal R}$ such that $G^{(0)}_n\neq 0$, then the  solution of the boundary value problem~\eqref{lS},~\eqref{ls.ic}--\eqref{ls.bc} satisfies
    \be
        \|u(\cdot,t)\|_{\Lebesgue^\infty(0,1)}\rightarrow_{t\to\infty}\infty,
    \ee
    and hence it is not asymptotically periodic.

    Otherwise, $u(x, t)$ is a periodic function of period $\frac 2 \pi\times  \max(1,\alpha)$.
\end{theorem}

\begin{rmk}
    In~\cite{Duj2009a}, the factor $\max(1,-\alpha/\beta)$ appears instead of $\max(1,\alpha)$. We have corrected this oversight.
\end{rmk}

Finally, if the time period and the spatial interval length are independent over $\QQ$, the solution cannot in general be asymptotically periodic.

\begin{theorem} \label{djt3}
    Let $u(x,t)$ denote the solution of the boundary value problem~\eqref{lS},~\eqref{ls.ic}--\eqref{ls.bc}.
    Assume that the period $T$ of the boundary datum $g_0(t)$ and the factor $2/\pi$
    are linearly independent over $\QQ$. Then  for all $x\in(0,1)$ the function $t\to u(x,t)$ is not asymptotically periodic.
\end{theorem}

\subsection{Revisiting the existing results using the \texorpdfstring{$Q$}{Q} equation approach} \label{ssec:ls.revisiting}
We aim to reproduce parts of Dujardin's theorems~\ref{djt1},~\ref{djt2} and~\ref{djt3} using our complex spectral approach.
To prove the  results stated above, Dujardin exploits classical tools in real analysis to decompose the Fourier series explicitly in several components, and study their specific continuity and asymptotic properties.
In contrast, we will introduce a general, algorithmic decomposition of the solution in two periodic parts with different periods.

Suppose that we are studying an IBVP with given, sufficiently regular, initial condition
\be\label{lSIC}
u(x,0)=u_0(x),
\qquad x\in(0,1),
\ee
and Dirichlet boundary conditions with asymptotically periodic data
\be \label{eqn:BCForPropNewDJT2}
    u(0,t)=g_0(t),\qquad u(1,t)=h_0(t).
\ee
We assume that the solution $u$ is strongly asymptotically periodic.

We start  by considering an example to illustrate the nature of the constraints that may arise.
\begin{eg} \label{exa1}
    Let the given boundary conditions be
    \[
        g_0(t)=u(0,t)=\sin(\pi^2 t), \quad h_0(t)=u(1,t)=0, \qquad t\geq0.
    \]
    Then the given data are $\frac 2 \pi$-periodic.
    Let $\omega=2\pi/T=\pi^2$, so that
    \[
        g_0(t)=\sum_{n\in\Z} G^{(0)}_n\re^{in\omega t}, \qquad
        h_0(t)=\sum_{n\in\Z} H^{(0)}_n\re^{in\omega t}, \qquad H^{(0)}_n=0.
    \]
    Under our assumption that $u$ is strongly asymptotically time periodic with period $T$, so is $Q$ and, asymptotically, $Q$, $g_1$ and $h_1$ can be represented as Fourier series~\eqref{eqn:AsymptoticBdryCoeffs}.
    The spectral equation~\eqref{eqn:Q} implies
    \be\label{qrel}
        q_n(k)=\frac{(ik G^{(0)}_n+G^{(1)}_n)-\re^{-ik}(ik H^{(0)}_n+H^{(1)}_n)}{k^2+n\omega}.
    \ee
    The coefficients $G_n^{(0)}$ and $H_n^{(0)}$ are known,
    \be
        G_n^{(0)}=\begin{cases}
        0&\abs{n}\neq 1,
        \\
        -\frac i 2& n=1,
        \\
        \frac i 2& n=-1,
        \end{cases}\qquad H_n^{(0)}=0.
        \label{Fn0}
    \ee
    From these and the condition of analyticity on $Q$ we can reconstruct the non-zero coefficients $G_n^{(1)}$ and $H_n^{(1)}$ using formulae~\eqref{eqn:Neumannrelations+} and~\eqref{eqn:Neumannrelations-}:
    \begin{align*}
        G_1^{(1)} &=
        -\frac {i\pi} {2\sinh \pi}\cosh \pi,
        &
        H_1^{(1)} &=
        -\frac {i\pi} {2\sinh \pi},
        \\
        G_{-1}^{(1)} &=
        \frac {i\pi} {2\sin \pi}\cos \pi,
        &
        H_{-1}^{(1)} &=
        \frac {i\pi} {2\sin \pi}.
    \end{align*}
  These expressions show that $G_{-1}^{(1)}$ and $H_{-1}^{(1)}$ are not well defined.
    Indeed, the assumption of strong asymptotic periodicity of $u$ leads to a contradiction, and we conclude that it is not possible to represent the solution as a time-periodic function of period $T$, even in an asymptotic sense.
\end{eg}
Example~\ref{exa1} is consistent with the statement of theorem~\ref{djt1}, and it illustrates the general result we now state, which is a slight generalisation of the first paragraph of Dujardin's theorem~\ref{djt2}.

\begin{prop} \label{prop:NewDJT2}
    Suppose $u$ satisfies the linear Schr\"odinger equation~\eqref{lS} and boundary conditions~\eqref{eqn:BCForPropNewDJT2} for known asymptotically $T$-periodic functions $g_0(t)$ and $h_0(t)$, in which
    \[
        g_0(t) = \sum_{n\in\ZZ} G^{(0)}_n \re^{\ri n \omega t} + o(1), \qquad\qquad
        h_0(t) = \sum_{n\in\ZZ} H^{(0)}_n \re^{\ri n \omega t} + o(1),\qquad \omega=\frac{2\pi}T.
    \]
    If there exist positive integers $(m,n)$ for which $\sqrt{n\omega}=m\pi$ and
    \[
        G^{(0)}_{-n} - (-1)^m H^{(0)}_{-n} \neq 0,
    \]
    then $u$ cannot be strongly asymptotically periodic.
\end{prop}

\begin{proof}
    Suppose $u$ is strongly asymptotically time periodic with period $T$.
By hypothesis, both sequences $(G_{n}^{(0)})_{n\in\ZZ}$ and $(H_{n}^{(0)})_{n\in\ZZ}$ are known.
We seek the sequences $(G_{n}^{(1)})_{n\in\ZZ}$ and $(H_{n}^{(1)})_{n\in\ZZ}$ which represent the periodic (nondecaying) parts of the Neumann boundary values.
Equations~\eqref{eqn:LS.qnnumerator.n0} provide explicit formulae for $G_0^{(1)}$ and $H_0^{(1)}$.
For each $n\in\NN$, equation~\eqref{eqn:LS.qnnumerator.+n} yields the following linear system:
\begin{equation} \label{eqn:LS.Dirichlet.LinearSystem.+}
    \begin{pmatrix} 1 & - \re^{\sqrt{n\omega}} \\ 1 & - \re^{-\sqrt{n\omega}} \end{pmatrix}
    \begin{pmatrix} G^{(1)}_{n} \\ H^{(1)}_{n} \end{pmatrix}
    =
    \begin{pmatrix}\sqrt{n\omega} \left( G^{(0)}_{n} - \re^{\sqrt{n\omega}}H^{(0)}_{n} \right) \\ -\sqrt{n\omega} \left( G^{(0)}_{n} - \re^{-\sqrt{n\omega}}H^{(0)}_{n} \right) \end{pmatrix}.
\end{equation}
Solving this system, we find
\begin{subequations} \label{eqn:Neumannrelations+}
\begin{align}
    \label{eqn:Neumannrelations.G+}
    G^{(1)}_{n} &= \frac{\sqrt{n\omega}}{\sinh\sqrt{n\omega}} \left( -H^{(0)}_{n} + G^{(0)}_{n} \cosh\sqrt{n\omega}\right), \\
    \label{eqn:Neumannrelations.H+}
    H^{(1)}_{n} &= \frac{\sqrt{n\omega}}{\sinh\sqrt{n\omega}} \left(  G^{(0)}_{n} - H^{(0)}_{n} \cosh\sqrt{n\omega}\right).
\end{align}
\end{subequations}
Similarly, for $n\in\NN$, using~\eqref{eqn:LS.qnnumerator.-n}, we find the linear system
\begin{equation} \label{eqn:LS.Dirichlet.LinearSystem.-}
    \begin{pmatrix} 1 & - \re^{-\ri\sqrt{n\omega}} \\ 1 & - \re^{\ri\sqrt{n\omega}} \end{pmatrix}
    \begin{pmatrix} G^{(1)}_{-n} \\ H^{(1)}_{-n} \end{pmatrix}
    =
    \begin{pmatrix}-\ri\sqrt{n\omega} \left( G^{(0)}_{-n} - \re^{-\ri\sqrt{n\omega}}H^{(0)}_{-n} \right) \\ \ri\sqrt{n\omega} \left( G^{(0)}_{-n} - \re^{\ri\sqrt{n\omega}}H^{(0)}_{-n} \right) \end{pmatrix}.
\end{equation}
Solving this system we find
\begin{subequations} \label{eqn:Neumannrelations-}
\begin{align}
    \label{eqn:Neumannrelations.G-}
    G^{(1)}_{-n} &= \frac{\sqrt{n\omega}}{\sin\sqrt{n\omega}} \left( -H^{(0)}_{-n} + G^{(0)}_{-n} \cos\sqrt{n\omega}\right), \\
    \label{eqn:Neumannrelations.H-}
    H^{(1)}_{-n} &= \frac{\sqrt{n\omega}}{\sin\sqrt{n\omega}} \left(  G^{(0)}_{-n} - H^{(0)}_{-n} \cos\sqrt{n\omega}\right).
\end{align}
\end{subequations}
    It follows that, as $t\to\infty$, the Neumann boundary values can be expressed as asymptotically valid Fourier series
    \[
        u_x(0,t) = \sum_{n\in\ZZ} G^{(1)}_n \re^{\ri n \omega t} + o(1), \qquad\qquad
        u_x(1,t) = \sum_{n\in\ZZ} H^{(1)}_n \re^{\ri n \omega t} + o(1),
    \]
    and the negatively indexed coefficients of those series satisfy the linear system~\eqref{eqn:LS.Dirichlet.LinearSystem.-}. Indeed, formulae~\eqref{eqn:Neumannrelations+} must be valid for all $n\in\NN$. In particular, it must hold that
$$
\sin\sqrt{n\omega}\neq 0 \,\,i.e.\,\, \sqrt{n\omega}\neq m\pi,\qquad m\in\mathbb Z.
$$

  The hypothesis in the statement of the  proposition implies that there exists a natural number $n$ for which the system is not full rank, hence does not have a unique solution. Therefore we cannot find a solution to the Dirchlet-to-Neumann map under this hypothesis, and we deduce a contradiction with the assumption of strong asymptotic time periodicity of $u$ with period $T$.
\end{proof}

\begin{rmk}

Strictly speaking, we have only proved that the solution $u$ cannot be strongly asymptotically $T$-periodic.
Since we assume that $T$ is the fundamental common period of the nondecaying parts of $g_0$ and $h_0$,
if $u$ is strongly asymptotically periodic, its period must be an integer multiple of $T$: $MT$, say.
    But then $\omega\mapsto\omega/M$, and $G_n^{(j)}$, $H_n^{(j)}$ are reindexed accordingly, with new $G_n^{(0)}=0$, $H_n^{(0)}=0$ inserted wherever $M\nmid n$.
    The reindexing preserves the product $n\omega$ in the denominator of equation~\eqref{qrel} and, by the above argument, we can still conclude that $u$ is not strongly asymptotically periodic with period $MT$.
    
\end{rmk}

The application of our approach to problems with periodic boundary conditions requires an assumption of (strong asymptotic) time periodicity of $u$.
Using this approach, we can only hope to derive necessary  conditions for (asymptotic) periodicity. To  study sufficient conditions to guarantee such periodicity, we need to add to this the analysis of an explicit representation of the solution.

However, it is important to note that our  approach derives necessary  conditions without the need for such explicit representation.
Moreover, the spectral equation~\eqref{Qeqnls} is independent of the initial condition.

It turns out that the initial condition plays an important role in determining the time dependence of the solution, and in particular it can affect its time periodicity. To illustrate this, we consider the homogeneous Dirichlet problem for the linear Schr\"odinger equation, and we use the formulae~\eqref{eqn:Neumannrelations.G+} derived from  equation~\eqref{Qeqn}  to prove that there is a specific initial condition for which the solution of the problem is time periodic with period exactly equal to $2/\pi$. This, combined with the explicit sine series solution representation, will give the full characterisation of time-periodicity.

\begin{prop} \label{prop:LSPeriodic}
    Consider equation~\eqref{lS}, with boundary conditions~\eqref{eqn:BCForPropNewDJT2}, with given initial condition \eqref{lSIC}  and $T$-periodic boundary data $g_0$ and $h_0$.

    Unless $u(x,0)=u_T(x)$, where $u_T$ is given in \eqref{heatuT} below, the solution of this problem can only be periodic if $T$ and $2/\pi$ are linearly dependent over $\mathbb Z$. In this case, the period is the least common multiple of $T$ and $2/\pi$.
\end{prop}

The proof we give provides the template for the proof in all other cases we consider, and has two constructive steps. We first assume that $u$ is asymptotically periodic, and that the necessary condition for periodicity characterised in proposition~\ref{prop:NewDJT2} hold, so that the systems \eqref{eqn:LS.Dirichlet.LinearSystem.+} and \eqref{eqn:LS.Dirichlet.LinearSystem.-} can be uniquely solved, determining the Dirichlet-to-Neumann map.
In the second step,  we construct an explicit initial datum $u_T$ with the property that the solution $u_1(x,t)$ of the boundary value problem, with given $T$ periodic Dirichlet boundary conditions and $u_T$ as initial condition, is {\em exactly} $T$-periodic.
We then show that the solution $u_2(x,t)$ of the homogeneous Dirichlet problem for this equation is time-periodic of period $\frac 2 \pi$.
Finally, we use the linearity of the problem to write the solution of the problem with generic initial datum $u_0$ and time periodic boundary conditions in the form
$$
u(x,t)=u_1(x,t)+u_2(x,t).
$$

\begin{proof}[Proof of proposition~\ref{prop:LSPeriodic}.]

Assuming that there exists a $T$-periodic solution $u_1(x,t)$ of the linear Schr\"odinger equation with $T$-periodic Dirichlet data, the Neumann boundary values $(u_1)_x(0,t)$ and $(u_1)_x(1,t)$ may be explicitly reconstructed.
Moreover, equation~\eqref{qnls} provides expressions for $q_n$, whence $Q$ can be found. Employing the inverse Fourier transform yields
\begin{align}
    u_1(x,t) &= G_0^{(0)} + (H_0^{(0)}-G_0^{(0)})x + \sum_{n\in\ZZ\setminus\{0\}} \re^{\ri n\omega t} U_n(x), \nonumber \\
    U_n(x) &= G_n^{(0)} \cos(\sqrt{n\omega} x) + \frac{H_n^{(0)}-\cos(\sqrt{n\omega})G_n^{(0)}}{\sin(\sqrt{n\omega})} \sin(\sqrt{n\omega} x).\label{Un}
\end{align}
The function $u_1(x,t)$ takes the following values at $t=0$:
\be\label{heatuT}
u_T(x)=G_0^{(0)}+(H_0^{(0)}-G_0^{(0)})x+\sum_{n\in\mathbb Z\setminus 0}U_n(x).
\ee
Therefore, if we pose the problem with $u_T$ as initial condition and the given boundary conditions, the formula \eqref{Un} defines a function that satisfies the PDE, as well as the initial and boundary conditions.
We conclude that the solution $u_1(x,t)$ of the problem with initial condition $u(x,0)=u_T(x)$ and $T$-periodic Dirichlet boundary conditions is $T$-periodic.

Now we consider the case that the given boundary conditions are homogeneous, $g_0(t)=h_0(t)=0$, and the initial condition is  $u_0(x)-u_T(x)$, with $u_0$ the given initial condition of the original problem, and $u_T$ defined above in \eqref{heatuT}. In this case, a series representation of the solution can be given in the form
$$
u_2(x,t)=\frac i 4\sum_{m=1}^\infty \hat u(m) \sin(m\pi x)\re^{-im^2\pi^2 t},
\qquad
\hat u(m)=\int_0^1 (u_0(x)-u_T(x))\re^ {-im\pi x}
$$
In particular, the time dependence is given through the functions $\exp(-\ri m^2\pi^2 t)$, so the solution $u_2(x,t)$ of this problem is time periodic with period $2/\pi$.

Finally, observe that since $u_T$ is not identically zero, we can write the solution of any given boundary value problem with given initial condition $u(x,0)=u_0(x)$ and  $T$-periodic given boundary conditions as the superposition of the solution of a homogeneous Dirichlet problem for the linear Schr\"odinger equation with initial datum $u_0-u_T$, and the solution of the problem with initial value $u_T$ and $T$-periodic Dirichlet boundary conditions.

Therefore, the solution of the linear Schr\"odinger equation with $T$-periodic Dirichlet data and any  initial condition can be expressed as the sum of a function of period $T$ and a function of period $2/\pi$, hence periodic with period given by the least common multiple of $T$ and $2/\pi$.
If $T$ and $\frac 2 {\pi}$ are not linearly dependent on $\mathbb Q$, the overall sum cannot be periodic.
\end{proof}

\section{The heat equation}

In this section, we study the heat equation~\eqref{heat} with a given initial condition and time periodic boundary conditions.
Although this equation is  well studied using classical methods, we analyse it here in order to  provide another example of the usefulness of our approach, and to contrast the results in this case with the case of the linear Schr\"odinger equation.
Indeed,  for the linear Schr\"odinger equation Proposition~\ref{prop:LSPeriodic} implies that, given Dirichlet boundary conditions, there exist cases when the solution cannot be periodic, or even asymptotically periodic. On the other hand, for the heat equation with given periodic boundary conditions, we expect that  the solution will usually be strongly asymptotically periodic with the same period as the boundary data. We prove this below for the case of Neumann boundary conditions, satisfying an  additional constraint, see the statement of Proposition \ref{prop:Heat.NeumanntoDirichlet} below. Moreover, in proving this proposition we provide the specific initial conditions that ensure that the solution is exactly periodic.  We conjecture that this case illustrates the generic picture for this PDE, when given Robin-type time-periodic boundary conditions.

\smallskip
Assuming that $u$ is strongly asymptotically periodic with period $T=2\pi/\omega$,  equation~\eqref{eqn:qn} becomes
\be\label{qnheat}
    q_n(k) = \frac{\left(\ri k G^{(0)}_n+G^{(1)}_n\right) - \re^{-\ri k}\left(\ri k H^{(0)}_n+H^{(1)}_n\right) }{\ri n\omega+k^2}, \qquad n\in\ZZ,
\ee
where the asymptotic Fourier coefficients are those in equations~\eqref{eqn:AsymptoticBdryCoeffs}.
For each $n\in\NN$, $q_{\pm n}$ are entire functions of $k\in\C$, so the system~\eqref{eqn:qnnumerator.n} yields the following equations:
\begin{subequations}
\label{eqn:Heat.qnnumerator}
\begin{align}
    \label{eqn:Heat.qnnumerator.+n}
    0 &= \left(\ri k G^{(0)}_n+G^{(1)}_n\right) - \re^{-\ri k}\left(\ri k H^{(0)}_n+H^{(1)}_n\right) & k &= \pm\ri\sqrt{\ri n\omega}, \\
    \label{eqn:Heat.qnnumerator.-n}
    0 &= \left(\ri k G^{(0)}_{-n}+G^{(1)}_{-n}\right) - \re^{-\ri k}\left(\ri k H^{(0)}_{-n}+H^{(1)}_{-n}\right) & k &= \pm\sqrt{\ri n\omega}.
\end{align}
Because $q_0$ is also entire, system~\eqref{eqn:qnnumerator.0} also holds:
\begin{equation} \label{eqn:Heat.qnnumerator.n0}
    G_0^{(1)}=H_0^{(1)}, \qquad\qquad H_0^{(0)}-G_0^{(0)} = H_0^{(1)}.
\end{equation}
\end{subequations}

To illustrate the use of our approach in this classical setting, we analyse the example of the  Neumann problem.

\begin{prop} \label{prop:Heat.NeumanntoDirichlet}
    Suppose that $u$ satisfies the heat equation \eqref{heat} with periodic Neumann boundary conditions
    \[
        u_x(0,t) = g_1(t) = \sum_{n\in\ZZ} G_n^{(1)} \re^{\ri n \omega t},
        \qquad
        u_x(1,t) = h_1(t) = \sum_{n\in\ZZ} H_n^{(1)} \re^{\ri n \omega t},
    \]
    with common period $T=2\pi/\omega$ and satisfying the constraint $G_0^{(1)}=H_0^{(1)}$.
    Then $u$ is strongly asymptotically $T$-periodic. Moreover, its Dirichlet boundary values are given by
    \[
        u(0,t) = g_0(t) = \sum_{n\in\ZZ} G_n^{(0)} \re^{\ri n \omega t} + o(1),
        \qquad
        u(1,t) = h_0(t) = \sum_{n\in\ZZ} H_n^{(0)} \re^{\ri n \omega t} + o(1),
    \]
    where, for $n\in\NN$ and $u_0(x)=u(x,0)$, the following formulae are valid:
    \begin{subequations} \label{eqn:Heat.NeumanntoDirichlet.Dirichlet}
    \begin{align} \label{eqn:Heat.NeumanntoDirichlet.Dirichlet.+n}
        G_n^{(0)} &= \frac{\csch\sqrt{\ri n\omega} H_n^{(1)} - \coth\sqrt{\ri n\omega} G_n^{(1)}}{\sqrt{\ri n \omega}},
        &
        H_n^{(0)} &= \frac{\coth\sqrt{\ri n\omega} H_n^{(1)} - \csch\sqrt{\ri n\omega} G_n^{(1)}}{\sqrt{\ri n \omega}},
        \\
        \label{eqn:Heat.NeumanntoDirichlet.Dirichlet.-n}
        G_{-n}^{(0)} &= \frac{\cot\sqrt{\ri n\omega} G_{-n}^{(1)} - \csc\sqrt{\ri n\omega} H_{-n}^{(1)}}{\sqrt{\ri n \omega}},
        &
        H_{-n}^{(0)} &= \frac{\csc\sqrt{\ri n\omega} G_{-n}^{(1)} - \cot\sqrt{\ri n\omega} H_{-n}^{(1)}}{\sqrt{\ri n \omega}},
        \\
        \label{eqn:Heat.NeumanntoDirichlet.Dirichlet.0}
        G_0^{(0)} &= \int_0^1 u_0(x)\D x - \frac{1}{2}G_0^{(1)},
        &
        H_0^{(0)} &= \int_0^1 u_0(x)\D x + \frac{1}{2}G_0^{(1)}.
    \end{align}
    \end{subequations}
    
\end{prop}

\begin{proof}
This proposition has an elementary proof based on the analysis of the Fourier series solution representation, but we provide a different proof via the Fourier-tranformed equation~\eqref{Qeqn}.

We begin with a proof of equations~\eqref{eqn:Heat.NeumanntoDirichlet.Dirichlet} using equation~\eqref{Qeqnheat}  and assuming strong asymptotic periodicity of $u$.
Subsequently, we justify the latter assumption.

\smallskip
\noindent
{\bf Step 1}

    Under the assumptions of  proposition~\ref{prop:Heat.NeumanntoDirichlet} and assuming additionally that $u$ is strongly asymptotically $T$-periodic, equations~\eqref{eqn:Heat.qnnumerator.+n} yield the linear system
    \[
        \begin{pmatrix} -\sqrt{\ri n \omega} & \sqrt{\ri n \omega} \re^{\sqrt{\ri n\omega}} \\ \sqrt{\ri n \omega} & -\sqrt{\ri n \omega} \re^{-\sqrt{\ri n\omega}} \end{pmatrix}
        \begin{pmatrix} G^{(0)}_{n} \\ H^{(0)}_{n} \end{pmatrix}
        =
        \begin{pmatrix} - G_{n}^{(1)} + \re^{\sqrt{\ri n \omega}} H_{n}^{(1)} \\ - G_{n}^{(1)} + \re^{-\sqrt{\ri n \omega}} H_{n}^{(1)} \end{pmatrix}.
    \]
    This system is full rank for all $n\in\NN$, and formulae~\eqref{eqn:Heat.NeumanntoDirichlet.Dirichlet.+n} follow.
    Equations~\eqref{eqn:Heat.qnnumerator.-n} provide the full rank linear system
    \[
        \begin{pmatrix} \ri\sqrt{\ri n \omega} & -\ri\sqrt{\ri n \omega} \re^{-\ri\sqrt{\ri n\omega}} \\ -\ri\sqrt{\ri n \omega} & \ri\sqrt{\ri n \omega} \re^{\ri\sqrt{\ri n\omega}} \end{pmatrix}
        \begin{pmatrix} G^{(0)}_{-n} \\ H^{(0)}_{-n} \end{pmatrix}
        =
        \begin{pmatrix} - G_{-n}^{(1)} + \re^{-\ri\sqrt{\ri n \omega}} H_{-n}^{(1)} \\ - G_{-n}^{(1)} + \re^{\ri\sqrt{\ri n \omega}} H_{-n}^{(1)} \end{pmatrix}.
    \]
 Thus, formulae~\eqref{eqn:Heat.NeumanntoDirichlet.Dirichlet.-n} follow.
    The first of equations~\eqref{eqn:Heat.qnnumerator.n0} holds by assumption.
    The second of equations~\eqref{eqn:Heat.qnnumerator.n0} together with analysis of the steady state solution of the Neumann problem with $g_1(t)=G_0^{(1)}=H_0^{(1)}=h_1(t)$ for the heat equation yields the linear system
    \[
        \begin{pmatrix} \frac{1}{2} & \frac{1}{2} \\ -1 & 1 \end{pmatrix}
        \begin{pmatrix} G^{(0)}_{0} \\ H^{(0)}_{0} \end{pmatrix}
        =
        \begin{pmatrix} \int_0^1u_0(x)\D x \\ G_0^{(1)} \end{pmatrix}.
    \]
   Thus, equations~\eqref{eqn:Heat.NeumanntoDirichlet.Dirichlet.0} are found.
   Hence, no additional necessary conditions need to be satisfied  for periodicity to hold.

\smallskip
\noindent
{\bf Step 2}

    We now justify the assumption that $u$ is strongly asymptotically periodic.
    
    We begin by constructing  the specific problem of this form whose solution $u_1$ is not just strongly asymptotically periodic, but periodic.
    
    As with the case of the linear Schr\"odinger equation, suppose that there exists a solution $u_1(x,t)$ which is exactly $T$ periodic, with $T=2\pi/\omega$.
    Then, by the previous proof, formulae~\eqref{eqn:Heat.NeumanntoDirichlet.Dirichlet} hold, not only asymptotically, but exactly:
    \[
        u_1(0,t) = g_0(t) = \sum_{n\in\ZZ} G_n^{(0)} \re^{\ri n \omega t},
        \qquad
        u_1(1,t) = h_0(t) = \sum_{n\in\ZZ} H_n^{(0)} \re^{\ri n \omega t},
    \]
    in which $G_0^{(0)}$ is free and $H_0^{(0)}=G_0^{(0)}+H_0^{(1)}$.
    Therefore,
    \begin{align*}
        u_1(x,t) &= G_0^{(0)} + G_0^{(1)}x + \sum_{n\in\ZZ\setminus\{0\}} \re^{\ri n\omega t} U_n(x), \\
        U_n(x) &= \frac{G_n^{(1)}}{\sqrt{\ri n \omega}} \sinh(\sqrt{\ri n\omega} x) + \frac{H_n^{(1)}-\cosh(\sqrt{\ri n\omega})G_n^{(1)}}{\sqrt{\ri n\omega}\sinh(\sqrt{\ri n\omega})} \cosh(\sqrt{\ri n\omega} x).
    \end{align*}
    Both the heat equation and the boundary conditions are satisfied by such $u$, which is $T$-periodic, so the existence assumption was justified.
    The initial value of this problem is given by evaluating
    \[
        u(x,0) = G_0^{(0)} + G_0^{(1)}x + \sum_{n\in\ZZ\setminus\{0\}} U_n(x)=:u_T(x).
    \]
    Therefore, the problem with $T$-periodic Neumann boundary conditions and initial condition $u(x,0)=u_T(x)$ has the $T$-periodic solution $u_1(x,t)$.
    
    Now we consider the original problem with $T$-periodic Neumann boundary values and an arbitrary initial datum $u_0$, but do not assume strong asymptotic periodicity of $u$.
    We separate this problem into two.
    The first problem has the $T$-periodic Neumann boundary values but initial datum $u_T$, in which we choose $G_0^{(0)}$ to be defined by equation~\eqref{eqn:Heat.NeumanntoDirichlet.Dirichlet.0}.
    By the above argument, the solution $u_1$ of this problem is $T$-periodic.
    As the second problem, we prescribe  homogeneous Neumann boundary values and the initial condition $u(x,0)=u_0(x)-u_T(x)$. By construction, $\int_0^1 [u_0(x)-u_T(x)] \D x=0$.
    Separation of variables yields a solution $u_2$ decaying like $\exp(-\pi^2 t)$ to $0$, whose time derivative is decaying at the same rate to $0$.
    The principle of linear superposition implies that the solution of the problem with the initial condition $u_0(x)$ and the given periodic boundary conditions is given by $$
    u(x,t)=u_1(x,t)+u_2(x,t),\qquad  \lim_{t\to\infty}u(x,t)=u_1(x,t).
    $$
    Hence the original problem is strongly asymptotically periodic.
\end{proof}

\section{The Stokes equation} \label{sec:LKdV}

In this section, we study the Stokes equation~\eqref{lkdv} with time periodic boundary conditions. Assuming that  $u$ is strongly asymptotically periodic with period $T=2\pi/\omega$, equation~\eqref{eqn:qn} yields
\begin{equation} \label{qn.LKdV}
    q_n(k) = \frac{\left(k^2 G^{(0)}_{n} - \ri k G^{(1)}_{n} - G^{(2)}_{n}\right) - \re^{-\ri k}\left(k^2 H^{(0)}_{n} - \ri k H^{(1)}_{n} - H^{(2)}_{n}\right)}{\ri(n\omega-k^3)}.
\end{equation}
The asymptotic Fourier coefficients are given by equations~\eqref{eqn:AsymptoticBdryCoeffs}.
For each $n\in\NN$, define
\begin{equation} \label{eqn:LKdV.kn}
    k_n:= \sqrt[3]{n\omega}, \qquad k_{-n}:=-\sqrt[3]{n\omega},
\end{equation}
in which $\sqrt[3]\argdot$ represents the real cube root function.
For each $n\in\ZZ\setminus\{0\}$, $q_{n}$ is entire, so system~\eqref{eqn:qnnumerator.n} becomes
\begin{subequations}
\label{eqn:LKdV.qnnumerator}
\begin{gather}
    \label{eqn:LKdV.qnnumerator.+-n}
    0 = \left(k^2 G^{(0)}_{n} - \ri k G^{(1)}_{n} - G^{(2)}_{n}\right) - \re^{-\ri k}\left(k^2 H^{(0)}_{n} - \ri k H^{(1)}_{n} - H^{(2)}_{n}\right), \\
    \notag
    \alpha=\exp(2\pi\ri/3),\qquad k = \alpha^j k_n, \qquad j \in\{0,1,2\}.
\end{gather}
Because $q_0$ is also entire, system~\eqref{eqn:qnnumerator.0}  holds, yielding the following equations:
\begin{equation} \label{eqn:LKdV.qnnumerator.n0}
    G_0^{(2)}=H_0^{(2)}, \qquad H_0^{(1)}-G_0^{(1)} = H_0^{(2)}, \qquad 2\left(H_0^{(0)}-G_0^{(0)}\right)=H_0^{(1)}-H_0^{(2)}.
\end{equation}
\end{subequations}

\subsection{Boundary conditions that do not couple the ends of the interval} \label{ssec:Airy.Uncoupled}
Assume that the given boundary conditions are
\begin{equation} \label{eqn:LKdV.Decoupled.BC}
    u(0,t) = g_0(t), \qquad\qquad u(1,t) = 0, \qquad\qquad u_x(1,t) = 0,
\end{equation}
where the boundary condition $g_0(t)$ is periodic, with period $T=\frac{2\pi}{\omega}$.
Under these boundary conditions, equation~\eqref{eqn:LKdV.qnnumerator.+-n} simplifies to
\[
    0=\left.\left[\left(k^2 G^{(0)}_{n} - \ri k G^{(1)}_{n}-G^{(2)}_{n}\right) +\re^{-\ri k} H^{(2)}_{n}\right]\right\rvert_{k=\alpha^jk_n}.
\]
This yields the system
\begin{equation} \label{eqn:LKdV.Decoupled.BdrySystem}
    \begin{pmatrix}
        1 & 1 & \re^{-\ri k_n} \\
        1 & \alpha  & \re^{-\ri\alpha k_n} \\
        1 & \alpha^2 & \re^{-\ri\alpha^2 k_n}
    \end{pmatrix}
    \begin{pmatrix}
        - G_n^{(2)} \\  -\ri k_nG_n^{(1)} \\ H_n^{(2)}
    \end{pmatrix}
    =
    G_n^{(0)} k_n^2
    \begin{pmatrix}
        1 \\ \alpha^2 \\ \alpha
    \end{pmatrix},
\end{equation}
in which $k_n$ are defined by equation~\eqref{eqn:LKdV.kn}.
This system has determinant $\Delta(k_n)$, given by
\begin{equation} \label{eqn:LKdV.Uncoupled.Delta}
    \Delta(k) = (\alpha^2-\alpha) \left(\re^{-\ri k}+\alpha\re^{-\ri\alpha k}+\alpha^2 \re^{-\ri\alpha^2 k}\right).
\end{equation}

\begin{eg}
    For any $\omega>0$, let the given boundary conditions be
    \[
        g_0(t)=\sin(\omega t), \quad h_0(t)=0, \quad h_1(t) = 0, \qquad t\geq0.
    \]
    In this case, the coefficients $H_n^{(0)}$ and $H_n^{(1)}$ vanish,  while
    \begin{equation} \label{Gn0}
        G_n^{(0)} = \begin{cases}
            0 & |n|\neq 1,
            \\
            -\frac i 2 & n=1,
            \\
            \frac i 2 & n=-1.
        \end{cases}
    \end{equation}
    The only non-zero boundary values are $G_{\pm 1}^{(1)}$, $G_{\pm 1}^{(2)}$ and $H_{\pm 1}^{(2)}$.
    They solve the systems
    \[
        \begin{pmatrix}
            1 & 1 & \re^{\mp\ri\sqrt[3]{\omega}} \\
            1 & \alpha  & \re^{\mp\ri\alpha\sqrt[3]{\omega}} \\
            1 & \alpha^2 & \re^{\mp\ri\alpha^2\sqrt[3]{\omega}}
            
        \end{pmatrix}
        \begin{pmatrix}
            - G_{\pm1}^{(2)} \\ \mp\ri\sqrt[3]{\omega}G_{\pm1}^{(1)} \\ H_{\pm1}^{(2)}
        \end{pmatrix}
        =
       \mp\frac{\ri\omega^{2/3}}{2} \begin{pmatrix}
            1 \\ \alpha^2 \\ \alpha
        \end{pmatrix}.
    \]
    Hence the quantities
    \begin{align*}
        G_{\pm1}^{(1)} &= \frac{\sqrt[3]{\omega}\left( \re^{\mp\ri \sqrt[3]{\omega}} + \alpha \re^{\mp\ri \alpha^2 \sqrt[3]{\omega}} + \alpha \re^{\mp\ri \alpha^2 \sqrt[3]{\omega}} \right)}{2\left( \re^{\mp\ri \sqrt[3]{\omega}} + \alpha \re^{\mp\ri \alpha \sqrt[3]{\omega}} + \alpha^2 \re^{\mp\ri \alpha^2 \sqrt[3]{\omega}} \right)},
        \\
        G_{\pm1}^{(2)} &= \frac{\mp\ri\omega^{2/3}}{2},
        \\
        H_{\pm1}^{(2)} &= \frac{\mp3\ri \omega^{2/3}}{2\left( \re^{\mp\ri \sqrt[3]{\omega}} + \alpha \re^{\mp\ri \alpha \sqrt[3]{\omega}} + \alpha^2 \re^{\mp\ri \alpha^2 \sqrt[3]{\omega}} \right)}
    \end{align*}
    are well defined.
    It follows that there is no contradiction with the assumption that $u(x,t)$ is strongly asymptotically time periodic, with period $T=2\pi/\omega$.
\end{eg}
We now show that this is generically the case, namely we will not need to satisfy additional necessary conditions for periodicity.
\begin{prop} \label{prop:LKdV.Decoupled}
    Suppose that $u$ satisfies the Stokes equation~\eqref{lkdv} with $T$-periodic boundary conditions~\eqref{eqn:LKdV.Decoupled.BC}, in which, for $\omega=2\pi/T$
    \[
        g_0(t) = \sum_{n\in\ZZ} G_n^{(0)} \re^{\ri n \omega t}.
    \]
    Then $u$ is strongly asymptotically $T$-periodic, and the remaining boundary values are given by
    \begin{gather*}
        u_x(0,t)
        = \sum_{n\in\ZZ} G_n^{(1)} \re^{\ri n \omega t} + o(1), \qquad
        u_{xx}(0,t)
        = \sum_{n\in\ZZ} G_n^{(2)} \re^{\ri n \omega t} + o(1), \\
        u_{xx}(1,t)
        = \sum_{n\in\ZZ} H_n^{(2)} \re^{\ri n \omega t} + o(1),
    \end{gather*}
    where, for $n\in\ZZ\setminus\{0\}$, $G_n^{(1)}$, $G_n^{(2)}$ and $H_n^{(2)}$ are the unique solutions of the system~\eqref{eqn:LKdV.Decoupled.BdrySystem} and $G_0^{(1)}$, $G_0^{(2)}$ and $H_0^{(2)}$ are the solutions of system~\eqref{eqn:LKdV.qnnumerator.n0}.
\end{prop}

\begin{proof}
First, we use the  analysis of the spectral equation~\eqref{Qeqnstokes} to show that, under the assumption of strong asymptotic periodicity of $u$, the unspecified boundary values can be found effectively.
Subsequently, we prove strong asymptotic $T$-periodicity of $u$.

\smallskip
\noindent
{\bf Step 1:}

    Under the assumption that the solution of the problem is strongly asymptotically periodic, the Fourier coefficients of the boundary values satisfy the systems~\eqref{eqn:LKdV.Decoupled.BdrySystem} and~\eqref{eqn:LKdV.qnnumerator.n0}.
    The latter system is full rank for the boundary conditions~\eqref{eqn:LKdV.Decoupled.BC}.
    
    It is known~\cite[\S A]{Pel2005a} that there are countably many complex zeros of $\Delta$.
    The asymptotic distribution of these zeros is well understood and, in particular, they never coincide with the cube roots of nonzero real numbers \cite{Lan1929a}.
    Hence, for all $n\in\ZZ\setminus\{0\}$, $\Delta(k_n)\neq0$ implying that the corresponding system is full rank.
    It is therefore always possible to solve for well-defined unknown boundary values in terms of the given ones.

\smallskip
\noindent
{\bf Step 2:}

    We now justify the assumption that the solution is  asymptotically periodic.

    Define $U_0$ by
    \[
        U_0(x) = G_0^{(0)}(x^2-2x+1).
    \]
    For $n\in\ZZ\setminus\{0\}$, define $U_n$ by
    \begin{equation*}
        U_n(x) = \frac{1}{2\pi}\int_{-\infty}^\infty \re^{\ri kx} q_n(k) \D k, \qquad
        q_n(k) = \frac{\left(k^2 G^{(0)}_{n} - \ri k G^{(1)}_{n} - G^{(2)}_{n}\right) + \re^{-\ri k} H^{(2)}_{n}}{\ri(n\omega-k^3)},
    \end{equation*}
    in which $G_n^{(2)}$, $G_n^{(1)}$ and $H_n^{(2)}$ satisfy the system~\eqref{eqn:LKdV.Decoupled.BdrySystem}.
    Then, by construction, for each $n\in\ZZ$,
    \[
        \left[\ri n\omega + \frac{\D^3}{\D x^3}\right] U_n(x) = 0, \qquad U_n(0)=G_n^{(0)}, \qquad U_n(1) = 0, \qquad U'_n(1) = 0.
    \]
    Therefore, the function $u_1(x,t)$ given by
    \[
        u_1(x,t):=\sum_{n\in\ZZ}\re^{\ri n \omega t} U_n(x)
    \]
    is a $T$-periodic function satisfying the Stokes equation~\eqref{lkdv} and the boundary conditions~\eqref{eqn:LKdV.Decoupled.BC}.
    The initial value of $u_1(x,t)$ is
    \[
        u_T(x):= \sum_{n\in\ZZ} U_n(x).
    \]
    
    Now suppose $u(x,t)$ satisfies the Stokes equation~\eqref{lkdv}, the boundary conditions~\eqref{eqn:LKdV.Decoupled.BC}, and the initial condition $u(x,0)=u_0(x)$.
    Then $u=u_1+u_2$ where $u_1$ is the function defined above and $u_2$ satisfies the Stokes equation~\eqref{lkdv} with  homogeneous boundary conditions,
    \[
        u(0,t) = 0, \qquad\qquad u(1,t) = 0, \qquad\qquad u_x(1,t) = 0,
    \]
    and the initial condition $u_2(x,0)=u_0(x)-u_T(x)$.
    Then, by~\cite[theorem~4.2]{Pel2005a} and~\cite[theorem~1.3]{Smi2012a}, the solution $u_2$ satisfies~
    \begin{align*}
        2\pi u_2(x,t) &= \int_{-\infty}^\infty \re^{\ri k x + \ri k^3t}(\hat u_0(k)-\hat{u}_T(k)) \D k + \int_{\partial D^+} \re^{\ri k x + \ri k^3t} Z^+(k) \D k \\
        &\hspace{19em} + \int_{\partial D^-} \re^{\ri k (x-1) + \ri k^3t} Z^-(k) \D k \\
        &= - \int_{\partial E^+} \re^{\ri k x + \ri k^3t} Z^+(k) \D k - \int_{\partial E^-} \re^{\ri k (x-1) + \ri k^3t} Z^-(k) \D k,
    \end{align*}
    where
    \begin{gather*}
    \begin{aligned}
        Z^+(k) &=
        \frac{-\hat{w}_0(k)\left(\alpha^2\re^{-\ri\alpha^2k}+\alpha\re^{-\ri\alpha k}\right) + \left(\alpha\hat{w}_0(\alpha k)+\alpha^2\hat{w}_0(\alpha^2 k)\right)\re^{-\ri k}}{\re^{-\ri k}+\alpha\re^{-\ri\alpha k}+\alpha^2 \re^{-\ri\alpha^2 k}}, \\
        Z^-(k) &= \frac{\hat{w}_0(k) + \alpha\hat{w}_0(\alpha k) + \alpha^2\hat{w}_0(\alpha^2 k)}{\re^{-\ri k}+\alpha\re^{-\ri\alpha k}+\alpha^2 \re^{-\ri\alpha^2 k}},\qquad \hat w_0:=\hat u_0-\hat{u}_T,
    \end{aligned} \\
        D^\pm = \CC^\pm\cap\{k:\Re(-\ri k^3)<0\}, \qquad E^\pm = \CC^\pm\cap\{k:\Re(-\ri k^3)>0\},
    \end{gather*}
    and the boundaries are positively oriented.
    
    Fix $\epsilon>0$ small enough to be less than half the infimal separation of the zeros of $\Delta$.
    Then, for all $t\geq0$,
    \[
        \re^{\ri k^3t}Z^-(k) = \bigoh{k^{-1}} \mbox{ as } k\to \infty \mbox{ from within } \clos\left( E^- \setminus \bigcup_{\substack{\lambda:\abs{\la}>0,\\\Delta(\la)=0}} B(\la,\epsilon) \right).
    \]
    uniformly in $\arg(\la)$.
    Hence, by Jordan's lemma,
    \begin{equation*}
        - \int_{\partial E^-} \re^{\ri k (x-1) + \ri k^3t} Z^-(k) \D k
        = - \sum_{\substack{\lambda:\abs{\la}>0,\\\Delta(\la)=0}} \int_{C(\la,\epsilon)} \re^{\ri k (x-1) + \ri k^3t} Z^-(k) \D k
        = \bigoh{\re^{\ri\la_0^3t}},
    \end{equation*}
    for $\la_0$ the smallest nonzero zero of $\Delta$ lying on the negative imaginary axis; a numerical root finder based on the argument principle reveals $\ri\la_0^3<-64$.
    
    The integrand in the integral about $\partial E^+$ does not exhibit such decay to the left of the contour of integration, so no such contour deformation argument may be applied.
    Moreover, it may not be deformed (outside an arbitrarily large finite region) away from $\partial E^+$, which is not a descent contour, so the method of steepest descent is not applicable.
    Rather, it is a Fourier integral, which may be treated using the method of stationary phase to determine the contribution from the degenerate critical point at $k=0$.
    Applying the method presented in~\cite[\S2.9]{Erd1956a}, we find that, for all $x$, the leading order contribution is $0$:
    \[
        \int_{\partial E^+}\re^{\ri kx+\ri k^3t} Z^+(k)\D k = 0t^{-1/3} + \bigoh{t^{-2/3}}.
    \]
    Therefore, $u_2(x,t)=\bigoh{t^{-2/3}}$.
    Similarly, $\partial_t u_2(x,t) = \bigoh{t^{-4/3}}$.
    
    Because $u_1$ is periodic and $u_2$ and its temporal derivative both decay, $u$ is strongly asymptotically periodic.
\end{proof}

\begin{rmk} \label{rmk:WhyIntegralRep}
    In the above proof of proposition~\ref{prop:LKdV.Decoupled}, we employed the integral representation of the function $u_2$ obtained via the unified transform method, and performed a stationary phase analysis at a degenerate critical point.
    In the corresponding parts of our proofs of the analogous propositions~\ref{prop:Heat.NeumanntoDirichlet} and~\ref{prop:LKdV.Coupled}, we used series representations, which admit simpler asymptotic analysis.
    The reason for the integral representation of $u_2$ in the proof of proposition~\ref{prop:LKdV.Coupled} is that no spectral series representation of the solution exists.
    Indeed, the eigenfunctions of the differential operator
    \[
        L\phi=\phi''' \qquad \Dom(L) = \{\phi:\phi(0)=0,\,\phi(1)=0,\,\phi'(1)=0\}
    \]
    do not form a complete system, see~\cite{Jac1915a,Hop1919a, Pel2005a}.
\end{rmk}

\subsection{Boundary conditions that couple the ends of the interval} \label{ssec:LKdV.Coupled}
We now consider another class of boundary conditions, that couple the two ends of the interval.

Namely, for real $\beta$, let the given boundary conditions be
\begin{equation} \label{eqn:LKdV.Coupled.BC}
    u(0,t) = g_0(t) = \sum_{k\in\ZZ} G_n^{(0)} \re^{\ri n \omega t}, \qquad u(1,t)=0, \qquad u_x(0,t) = \beta u_x(1,t),\quad \abs\beta\geq1,
\end{equation}
with $\omega=2\pi/T$.

Assuming that $u$ is strongly asymptotically $T$-periodic, system~\eqref{eqn:LKdV.qnnumerator.+-n} simplifies to
\[
    0 = \left.\left[k^2 G^{(0)}_{n} - G^{(2)}_{n} - \ri k\left(\beta-\re^{-\ri k}\right) H^{(1)}_{n} + \re^{-\ri k} H^{(2)}_{n}\right]\right\rvert_{k=\alpha^jk_n}, \qquad\qquad G^{(1)}_{n} = \beta H^{(1)}_{n}.
\]
Therefore, the system that characterises the unknown boundary values is, for $n\in\ZZ\setminus\{0\}$,
\begin{equation} \label{eqn:LKdV.Coupled.BdrySystem}
    \begin{pmatrix}
        1 &         \left(\beta-\re^{-\ri         k_n}\right) & \re^{-\ri k_n} \\
        1 & \alpha  \left(\beta-\re^{-\ri \alpha  k_n}\right) & \re^{-\ri\alpha k_n} \\
        1 & \alpha^2\left(\beta-\re^{-\ri \alpha^2k_n}\right) & \re^{-\ri\alpha^2 k_n}
    \end{pmatrix}
    \begin{pmatrix}
        - G_n^{(2)} \\  -\ri k_nH_n^{(1)} \\ H_n^{(2)}
    \end{pmatrix}
    =
    G_n^{(0)} k_n^2
    \begin{pmatrix}
        1 \\ \alpha^2 \\ \alpha
    \end{pmatrix}, \qquad G^{(1)}_{n} = \beta H^{(1)}_{n},
\end{equation}
where $k_n$ are defined by equation~\eqref{eqn:LKdV.kn}.
The determinant of this system is $\Delta(k_n)$, with $\Delta$ given by
\begin{equation} \label{eqn:LKdV.Coupled.Delta}
    \Delta(k) = (\alpha^2-\alpha) \sum_{j=0}^2 \alpha^j \left( \re^{\ri\alpha^jk} + \beta \re^{-\ri\alpha^j k} \right).
\end{equation}
In the arguments below, it will be useful to use the following formula:
\begin{multline} \label{eqn:LKdV.Coupled.Delta2}
    \Delta(k) = (\beta-1)\sqrt{3}\left( - \sin(k) + \sin\left(\frac{k}{2}+\frac{2\pi}{3}\right)\re^{\sqrt{3}k/2} + \sin\left(\frac{k}{2}-\frac{2\pi}{3}\right)\re^{-\sqrt{3}k/2} \right) \\
    - \ri(\beta+1)\sqrt{3}\left( \cos(k) - \sin\left(\frac{\pi}{6}+\frac{k}{2}\right)\re^{\sqrt{3}k/2} - \sin\left(\frac{\pi}{6}-\frac{k}{2}\right)\re^{-\sqrt{3}k/2} \right).
\end{multline}

\begin{rmk}
    We only consider the case that $\abs\beta\geq1$ because, when $\abs\beta<1$, the solution blows up instantaneously for generic initial data.
    Indeed, if $0<\beta<1$, then an analysis using the arguments of~\cite{Lan1931a} yields, for all $m\in\ZZ$,
    \[
        \lambda_m=(6\abs{m}-1)\pi/3 + \log(\abs\beta) + \bigoh{\re^{-\abs{m}\pi\sqrt3}}.
    \]
    The time evolution of the corresponding eigenfunctions is governed by $\exp(\ri\lambda_m^3t)$, and the real part of $\ri\lambda_m^3$ is unbounded above because $\log(\abs\beta)<0$.
    It follows that for any $t>0$ the solution is unbouded.
    A similar analysis applies for $-1<\beta<0$, but the real parts of $\lambda_m$ are shifted by $\pi$.
    The case $\beta=0$ is, after change of variables $x\mapsto1-x$, the time reversal of $\beta=\infty$, which was studied in~\S\ref{ssec:Airy.Uncoupled}.
    A similar argument shows that the $\beta=0$ problem is also ill-posed.
\end{rmk}

\begin{prop} \label{prop:LKdV.Coupled}
    Suppose that $u$ satisfies the Stokes equation~\eqref{lkdv} with the $T$-periodic boundary conditions~\eqref{eqn:LKdV.Decoupled.BC}, where the real number $\beta$ obeys $\abs\beta>1$.
    Then $u$ is strongly asymptotically $T$-periodic.  Its unspecified boundary values satisfy the following estimates:
    \begin{gather*}
        u_x(0,t)
        = \sum_{n\in\ZZ} G_n^{(1)} \re^{\ri n \omega t} + o(1),
        \qquad
        u_{xx}(0,t)
        = \sum_{n\in\ZZ} G_n^{(2)} \re^{\ri n \omega t} + o(1), \\
        u_x(1,t)
        = \sum_{n\in\ZZ} H_n^{(1)} \re^{\ri n \omega t} + o(1),
        \qquad
        u_{xx}(1,t)
        = \sum_{n\in\ZZ} H_n^{(2)} \re^{\ri n \omega t} + o(1),
    \end{gather*}
    where, for $n\in\ZZ\setminus\{0\}$, $G_n^{(1)}$, $G_n^{(2)}$, $H_n^{(1)}$ and $H_n^{(2)}$ are the solutions of the full rank system~\eqref{eqn:LKdV.Coupled.BdrySystem}, and $G_0^{(1)}$, $G_0^{(2)}$, $H_0^{(1)}$ and $H_0^{(2)}$ are the simultaneous solutions of the system~\eqref{eqn:LKdV.qnnumerator.n0} and $G_0^{(1)}=\beta H_0^{(1)}$.
\end{prop}

\begin{proof}
As before, we begin by using the analysis of equation~\eqref{Qeqnstokes} to determine the unknown boundary values under the assumption of strong asymptotic periodicity of $u$.
Then, we give the full proof.

\smallskip
\noindent
{\bf Step 1:}

Assuming that the solution is asymptotically periodic, the Fourier coefficients of the boundary values satisfy systems~\eqref{eqn:LKdV.Coupled.BdrySystem} and~\eqref{eqn:LKdV.qnnumerator.n0}.  For the decoupled boundary conditions~\eqref{eqn:LKdV.Decoupled.BC}, the
  system with $n=0$ reduces to
    \[
        \begin{pmatrix}
            1 & 0 & -1 \\
            0 & 1-\beta & -1 \\
            0 & 1 & -1
        \end{pmatrix}
        \begin{pmatrix} G_0^{(2)} \\ H_0^{(1)} \\ H_0^{(2)} \end{pmatrix}
        =
        \begin{pmatrix} 0 \\ 0 \\ -2G_0^{(0)} \end{pmatrix},
        \qquad
        G_0^{(1)} = \beta H_0^{(1)},
    \]
and  is of full rank.
    
    Suppose $\beta>1$  (the analysis for $\beta<-1$ is very similar and is omitted).
    Then, expanding formula~\eqref{eqn:LKdV.Coupled.Delta} into real and imaginary parts for $k\in\RR$, we obtain formula~\eqref{eqn:LKdV.Coupled.Delta2}.
    It can be determined numerically that $\Delta(k)\neq0$ for $0<k<10/\sqrt{3}$.
    If $k\geq10/\sqrt{3}$, then, for the real part of $\Delta(k)$ to vanish, it must be that
    \[
        \abs{\sin\left(\frac{k}{2}+\frac{2\pi}{3}\right)}<2\re^{-5} \qquad\Rightarrow\qquad k = 2m\pi - \frac{4\pi}{3}+\epsilon,\; \exists\, \epsilon<\frac{1}{24}.
    \]
    But then
    \begin{align*}
        \abs{\Im\Delta(k)} &= (\beta+1)\sqrt{3}\left\lvert \cos\left(\frac{4\pi}{3}-\epsilon\right) - \sin\left(m\pi-\frac{\pi}{2}+\frac{\epsilon}{2}\right)\re^{\sqrt{3}k/2} \right. \\
        &\hspace{15em} \left. - \sin\left(-m\pi-\frac{5\pi}{6}-\frac{\epsilon}{2}\right)\re^{-\sqrt{3}k/2} \right\rvert \\
        &\geq (\beta+1)\sqrt{3}\left( \frac{47\re^{5}}{48} - 2 \right) > 0.
    \end{align*}
    Therefore, there are no positive real zeros of $\Delta$.
    By equation~\eqref{eqn:LKdV.Coupled.Delta}, $\Delta$ satisfies the symmetry condition $\Delta(-k)=-\overline{\Delta(k)}$, where the bar denotes complex conjugation.
    It follows that there are no negative real zeros of $\Delta$ either.
    Because there are no nonzero real zeros of $\Delta$, the system~\eqref{eqn:LKdV.Coupled.BdrySystem} is full rank and hence admits a unique solution, regardless of $n\in\ZZ\setminus\{0\}$ and $\omega>0$.

\smallskip
\noindent
{\bf Step 2:}
    We now justify the assumption made above that the solution is asymptotically periodic.
    
    Define $U_0$ by
    \[
        U_0(x) = G_0^{(0)}(-x+1).
    \]
   For $n\in\ZZ\setminus\{0\}$, define $U_n$ by
    \begin{equation*}
        U_n(x) = \frac{1}{2\pi}\int_{-\infty}^\infty \re^{\ri kx} q_n(k) \D k, \qquad
        q_n(k) = \frac{k^2 G^{(0)}_{n} - G^{(2)}_{n} - \ri k\left(\beta-\re^{-\ri k}\right) H^{(1)}_{n} + \re^{-\ri k} H^{(2)}_{n}}{\ri(n\omega-k^3)},
    \end{equation*}
    in which $G_n^{(2)}$, $G_n^{(1)}$ and $H_n^{(2)}$ satisfy the system~\eqref{eqn:LKdV.Coupled.BdrySystem}.
    Then, by construction, for each $n\in\ZZ$ we have
    \[
        \left[\ri n\omega + \frac{\D^3}{\D x^3}\right] U_n(x) = 0, \qquad U_n(0)=G_n^{(0)}, \qquad U_n(1) = 0, \qquad U'_n(0) = \beta U'_n(1).
    \]
    Therefore,
    \[
        u_1(x,t):=\sum_{n\in\ZZ}\re^{\ri n \omega t} U_n(x)
    \]
    is a $T$-periodic function satisfying the Stokes equation~\eqref{lkdv} with the boundary conditions~\eqref{eqn:LKdV.Coupled.BC}.
    The initial value of $u_1$ is
    \[
        u_T(x):= \sum_{n\in\ZZ} U_n(x).
    \]
    
    As in the previous case, suppose $u(x,t)$ satisfies the Stokes equation~\eqref{lkdv}, the boundary conditions~\eqref{eqn:LKdV.Coupled.BC}, and the initial condition $u(x,0)=u_0(x)$.
    Then $u=u_1+u_2$ where $u_1$ is the function defined above and $u_2$ satisfies the Stokes equation~\eqref{lkdv}, with the homogeneous boundary conditions
    \[
        u(0,t) = 0, \qquad\qquad u(1,t) = 0, \qquad\qquad u_x(0,t) = \beta u_x(1,t),
    \]
    and the initial condition $u_0(x)-u_T(x)$.
    
    The differential operator $L$, defined by
    \[
        L\phi = (-\ri)^3\phi''', \qquad \Dom(L) = \{\phi\in \AC^2[0,1]: \phi(0)=0,\,\phi(1)=0,\, \phi'(0)=\beta\phi'(1)\},
    \]
    is regular in the sense of~\cite{Loc2000a}. Thus, $L$ and its Lagrange adjoint
    \[
        L^\star\psi = (-\ri)^3\phi''', \qquad \Dom(L^\star) = \{\psi\in \AC^2[0,1]: \psi(0)=0,\,\psi(1)=0,\, \beta\psi'(0)=\psi'(1)\}
    \]
    have the property that their eigenfunctions $(E_m,E_m^\star)_{m\in\NN}$, with corresponding eigenvalues $\la_m^3$ satisfying the equations
    \begin{equation} \label{eqn:LKdV.Coupled.Eigs}
        LE_m=\la_m^3E_m, \qquad L^\star E_m^\star = \overline{\la_m^3} E_m^\star,
    \end{equation}
    form a biorthogonal basis.
    Therefore,
    \begin{equation} \label{eqn:LKdV.Coupled.SeriesSoln}
        u_2(x,t) = \sum_{m\in\NN} \left\langle u_0-u_T , E_m^\star \right\rangle \re^{-\ri \lambda_m^3t} E_m(x).
    \end{equation}
    Therefore, the asymptotic time analysis of $u_2(x,t)$ is reduced to the study of the  time behaviour of the functions $\exp(-\ri\lambda_m^3t)$.
    
    By the first of equations~\eqref{eqn:LKdV.Coupled.Eigs}, for some $A_m,B_m,C_m\in\CC$, we have
    \[
        E_m(x) = \begin{cases}A_mx^2+B_mx+C_m & \la_m=0, \\A_m \re^{\ri\la_m x} + B_m \re^{\ri\alpha\la_m x} + C_m \re^{\ri\alpha^2\la_m x} & \lambda_m\neq0.\end{cases}
    \]
    But if $\la_m=0$, then $E_m\in\Dom(L)$ implies $E_m=0$, so we discount the trivial eigenvalue.
    Because $E_m\in\Dom(L)$, for $\la_m\neq0$, we find
    \[
        \begin{pmatrix}
            1 & 1 & 1 \\
            \re^{\ri\la_m} & \re^{\ri\alpha\la_m} & \re^{\ri\alpha^2\la_m} \\
            \ri\la_m\left(1-\beta\re^{\ri\la_m}\right) & \ri\alpha\la_m\left(1-\beta\re^{\ri\alpha\la_m}\right) & \ri\alpha^2\la_m\left(1-\beta\re^{\ri\alpha^2\la_m}\right)
        \end{pmatrix}
        \begin{pmatrix} A_m \\ B_m \\ C_m \end{pmatrix}
        =
        \begin{pmatrix} 0 \\ 0 \\ 0 \end{pmatrix}.
    \]
    Row and column operations reveal that this system has nontrivial solutions if and only if $\Delta(\la_m)=0$, in which $\Delta$ is the determinant of system~\eqref{eqn:LKdV.Coupled.BdrySystem}.
    We have shown that the nonzero zeros of $\Delta$ are the cube roots of the eigenvalues of the spatial differential operator $L$.
    
    An asymptotic analysis using the arguments of~\cite{Lan1931a} shows that the exponential polynomial $\Delta$ has zeros asymptotically distributed on the lines
    \[
        \alpha^j\left( \ri\log\abs\beta + \RR \right),
    \]
    so there can be at most finitely many zeros of $\Delta$ in the closure of
    \[
        D:=\{k:\Re(-\ri k^3)<0\}.
    \]
    
    Therefore, the nonzero zeros of $\Delta$ have a finite infimal separation from $\clos(D)$.
    It follows that $\sup\{\Re(-\ri \la_m^3)\}<0$, so the terms in the series~\eqref{eqn:LKdV.Coupled.SeriesSoln} decay uniformly exponentially in time.
    Differentiating the series in~\eqref{eqn:LKdV.Coupled.SeriesSoln} term by term, we see that $\partial_t u_2(x,t)$ also decays exponentially in time.
    
   Since $u=u_1+u_2$ and $u_1$ is $T$-periodic whereas $u_2$ is uniformly exponentially decaying in time, it follows that $u$ is strongly asymptotically $T$-periodic.
\end{proof}

\begin{rmk}
    In the above proof, we have assumed that there no nonzero zeros of $\Delta$ in $\clos(D)$.
    While this appears to be the case numerically (see figure~\ref{fig:zerosDelta}), we have not given a rigorous proof.
    Therefore we have implicitly excluded form our consideration any case that may have finitely many zeros within $\clos(D)$.
\end{rmk}

\begin{figure}
    \centering\includegraphics[width=0.7\linewidth]{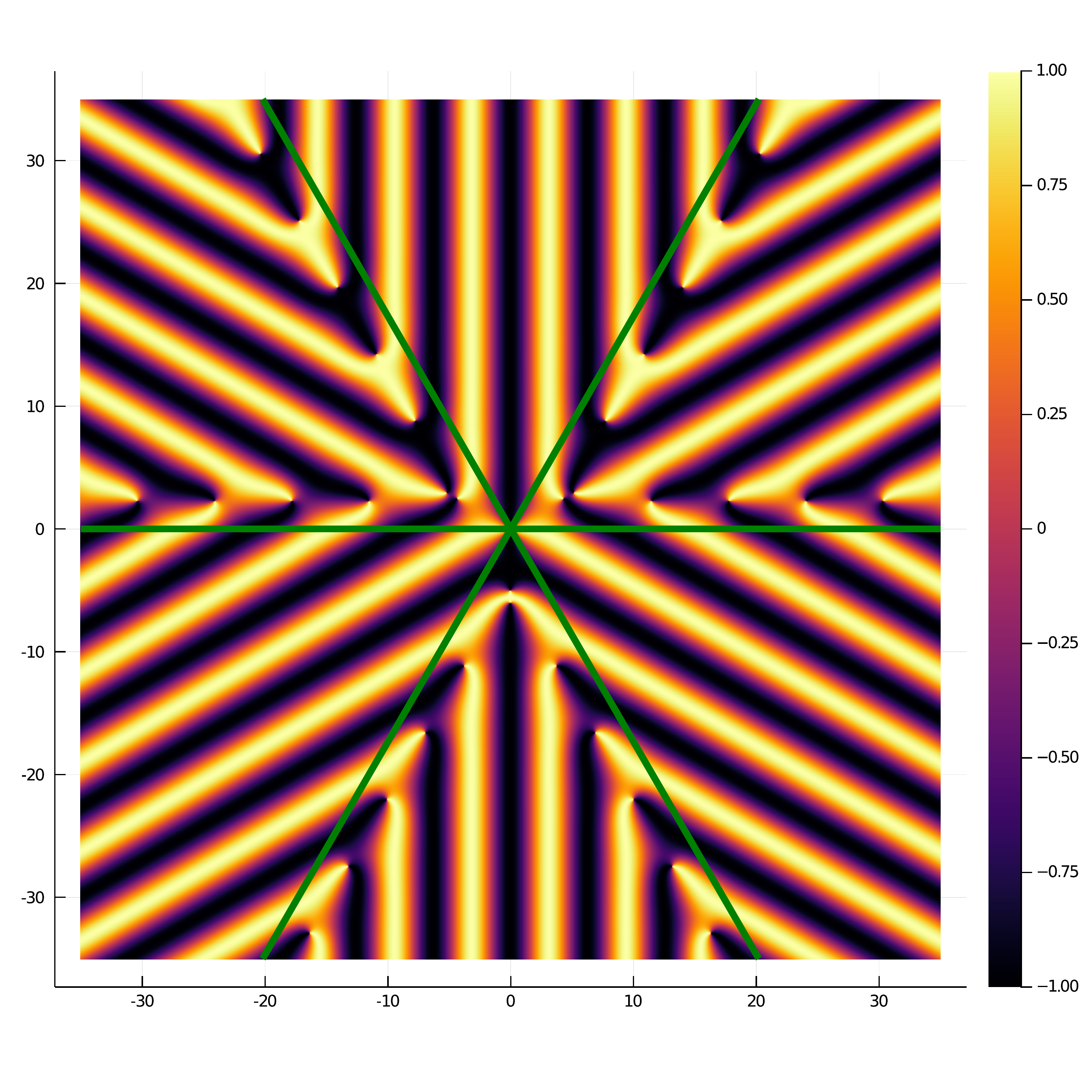}
    \caption{
        A plot showing the zeros of $\Delta$ in the case $\beta=10$.
        The plot is a heatmap of $\sin(\arg(\Delta(x+\ri y)))$.
        Because $\Delta$ is entire, the only points of discontinuity of this heat map are zeros of $\Delta$.
    }
    \label{fig:zerosDelta}
\end{figure}

Finally, we consider the special case when $\beta=\pm 1$, when, in contrast to the case just considered, the eigenvalues of the space operator are real.
Proposition~\ref{prop:LKdV.CoupledBeta1} should be compared and contrasted with proposition~\ref{prop:NewDJT2}.
The irregular (albeit asymptotically regular) separation of $\la_m$ in the Stokes case necessitates a different characterisation of the necessary condition for strong asymptotic periodicity.
Whereas in proposition~\ref{prop:NewDJT2} we were able to give a commensurability characterisation of the periodicity criterion, in proposition~\ref{prop:LKdV.CoupledBeta1} we cannot simplify beyond $\Delta(k_n)\neq0$.
Adapting proposition~\ref{prop:NewDJT2} to Robin boundary conditions such as $u(0,t)=g_0(t)$, $u_x(1,t)+bu(1,t)=h_0(t)$ would have a similar effect.
\begin{prop} \label{prop:LKdV.CoupledBeta1}
    Suppose that $u$ satisfies the Stokes equation~\eqref{lkdv} with the $T$-periodic boundary conditions~\eqref{eqn:LKdV.Coupled.BC} and $\beta=\pm1$:
    $$
     u(0,t) = g_0(t) = \sum_{n\in\ZZ} G_n^{(0)} \re^{\ri n \omega t}, \qquad\qquad u(1,t)=0, \qquad\qquad u_x(0,t) = \pm u_x(1,t).
     $$
     Let $k_n$ be defined by \eqref{eqn:LKdV.kn} and let $\Delta(k)$ be defined by \eqref{eqn:LKdV.Coupled.Delta2}.
     
   A necessary condition for the system to admit a well-defined strongly asymptotically periodic solution $u(x,t)$ is that for each $n\in\ZZ\setminus\{0\}$, $\Delta(k_n)\neq 0$ or $G_n^{(0)}=0$.
\end{prop}

\begin{proof}
    Suppose $u$ is the unique solution of the given problem which is asymptotically periodic, namely the criteria of the proposition apply, and the Fourier coefficients of the unknown boundary values satisfy system~\eqref{eqn:LKdV.Coupled.BdrySystem}.
    If $\beta=1$,  equation~\eqref{eqn:LKdV.Coupled.Delta2} for the determinant of this system simplifies to
    \[
        \Delta(k) = -\ri2\sqrt{3} \left( \cos(k) - \sin\left(\frac{\pi}{6}+\frac{k}{2}\right)\re^{\sqrt{3}k/2} - \sin\left(\frac{\pi}{6}-\frac{k}{2}\right)\re^{-\sqrt{3}k/2} \right).
    \]
    This formula shows that  $\Delta$ has infinitely many real zeros $\lambda_m$, asymptotically (in $m$) distributed like $\lambda_m\sim\pm(2m-1/3)\pi$.
    If $\beta=-1$,
    equation~\eqref{eqn:LKdV.Coupled.Delta2} simplifies to
    \[
        \Delta(k) = -2\sqrt{3}\left( - \sin(k) + \sin\left(\frac{k}{2}+\frac{2\pi}{3}\right)\re^{\sqrt{3}k/2} + \sin\left(\frac{k}{2}-\frac{2\pi}{3}\right)\re^{-\sqrt{3}k/2} \right),
    \]
    and a similar argument applies.
    Therefore, if $\omega$ is such that there is a nonzero integer $n$ for which $\Delta(k_n)=0$,  the system~\eqref{eqn:LKdV.Coupled.BdrySystem} is not of full rank. It follows that if there is an $n$ such that $\Delta (k_n)=0$, the Dirichlet to Neumann map cannot be uniquely determined.
\end{proof}

\begin{rmk}
    We note that even if the necessary conditions for periodicity given by Proposition~\ref{prop:LKdV.CoupledBeta1} are satisfied, the solution will not usually be asymptotically periodic.
    Indeed, in this case we can construct, as in the previous cases, a particular initial condition $u_T(x)$ that guarantees that the solution $u_1$ of the associated problem is exactly periodic of period $T=2\pi/\omega$.
    On the other hand, the homogeneous version of the problem with $u(x,0)=u_0(x)-u_T(x)$ has a solution $u_2$ which can be expressed in terms of the complete set of eigenfunctions $\re^{-i\la_m^3t}$.
    Since the $\la_m$ are real, each of these eigenfunctions is individually periodic, but each has a different period, and the individual periods are incommensurate (although they satisfy the asymptotic condition $\la_{m}-\la_{m-2}\to_{m\to\infty}2\pi$, $\la_m-\la_{m-2}\neq2\pi$.)
    It follows that $u_2$ is not a periodic function, and therefore neither is the solution $u=u_1+u_2$.
\end{rmk}

\section*{Conclusions}
In this paper, we have considered boundary value problems posed on a finite interval, for three important linear PDE in one spatial variable.
For the second order cases, we have given the general theory for either Dirichlet or Neumann conditions.
For the third order case, we have discussed several general classes of boundary conditions.

The crucial ingredient of our methodology is the idea, borrowed from the unified transform approach, that the spectral representation of the given PDE should be considered in a complex rather than real spectral space. This yields naturally necessary conditions under which the solution may be periodic. The general complex integral representation of the solution of 2-point linear  boundary value problem  is also used to derive our results for third order problems.

In all the problems considered we have assumed that the given boundary conditions have the same asymptotic period.
Then it is natural to assume that the solution, if it is periodic at all, will share the same period as the boundary conditions.

Such assumption is no longer obvious when the boundary conditions have different periods, either commensurate or not.
However, we expect that it will be possible to  split the solution into  parts that have the periodicity of each boundary conditions, and use an approach analogous to the one presented in this paper.
The generalisation to such cases is left for future work.

\bibliographystyle{amsplain}
\bibliography{dbrefs}

\end{document}

%% file: texhead-main.tex
\usepackage{amssymb,amsmath,amsthm}
\usepackage{xcolor,graphicx}
\usepackage{verbatim}
\usepackage{mathtools}
\usepackage{hyperref}
\usepackage{titling}
\usepackage{fancyhdr}
\newcommand{\runningtitle}{Running Title}
\newtheorem{thm}{Theorem}

\newtheorem{prop}[thm]{Proposition}

\theoremstyle{definition}
\newtheorem{defn}[thm]{Definition}

\theoremstyle{remark}
\newtheorem{rmk}[thm]{Remark}
\newtheorem{eg}[thm]{Example}

\newcommand{\NN}{\mathbb N}              % The set of naturals
\newcommand{\ZZ}{\mathbb Z}              % The set of integers
\newcommand{\QQ}{\mathbb Q}              % The set of rationals
\newcommand{\RR}{\mathbb R}              % The set of reals
\newcommand{\CC}{\mathbb C}              % The set of complex numbers
\renewcommand{\Re}{\operatorname*{Re}} \renewcommand{\Im}{\operatorname*{Im}}
\newcommand{\D}{\ensuremath{\,\mathrm{d}}}
\newcommand{\ri}{\ensuremath{\mathrm{i}}}
\newcommand{\re}{\ensuremath{\mathrm{e}}}
\newcommand{\la}{\ensuremath{\lambda}}
\renewcommand{\epsilon}{\varepsilon}
\renewcommand{\geq}{\geqslant}
\renewcommand{\leq}{\leqslant}
\providecommand{\csch}{\operatorname{csch}}

\providecommand{\clos}{\operatorname{clos}}
\newcommand{\abs}[1]{\left\lvert#1\right\rvert}

\usepackage{scalerel}
\usepackage{stackengine}
\stackMath
\newcommand\reallywidecheck[1]{%
\savestack{\tmpbox}{\stretchto{%
  \scaleto{%
    \scalerel*[\widthof{\ensuremath{#1}}]{\kern-.6pt\bigwedge\kern-.6pt}%
    {\rule[-\textheight/2]{1ex}{\textheight}}%WIDTH-LIMITED BIG WEDGE
  }{\textheight}%
}{0.5ex}}%
\stackon[1pt]{#1}{\scalebox{-1}{\tmpbox}}%
}

\newcommand\reallywidehat[1]{%
\savestack{\tmpbox}{\stretchto{%
  \scaleto{%
    \scalerel*[\widthof{\ensuremath{#1}}]{\kern-.6pt\bigwedge\kern-.6pt}%
    {\rule[-\textheight/2]{1ex}{\textheight}}%WIDTH-LIMITED BIG WEDGE
  }{\textheight}%
}{0.5ex}}%
\stackon[1pt]{#1}{\tmpbox}%
}

%% file: texhead-project.tex
\usepackage{mathrsfs}
\usepackage{nicematrix}
\numberwithin{equation}{section}

\providecommand{\bigoh}[1]{\mathcal{O}\left(#1\right)}

\newcommand{\AC}{\operatorname{AC}} % Absolutely continuous functions
\providecommand{\Lebesgue}{\mathrm{L}} % Upright L for Lebesgue spaces

\providecommand{\argdot}{{}\cdot{}}

\newsavebox{\accentbox}

\providecommand{\Dom}{\operatorname{Dom}} % Domain of an operator

\newcommand{\be}{\begin{equation}}
\newcommand{\ee}{\end{equation}}
\newcommand{\bea}{\begin{eqnarray}}
\newcommand{\eea}{\end{eqnarray}}
\newcommand{\barr}{\begin{array}}
\newcommand{\earr}{\end{array}}
\newcommand{\bpar}{\begin{equation} \left\{ \begin{array}{lll}}
\newcommand{\epar}{ \end{array}\right. \end{equation} }
\newcommand{\eparn}{ \end{array} \right.}

\newcommand\leftmat{\left(\begin{array}{cc}}
\newcommand\rightmat{\end{array}\right)}
\newcommand\leftvec{\left(\begin{array}{c}}
\newcommand\rightvec{\end{array}\right)}
\newcommand\R{{\mathbb R}}
\newcommand\C{{\mathbb C}}
\newcommand\Z{{\mathbb Z}}

\theoremstyle{plain}
\newtheorem{theorem}[thm]{Theorem}

\theoremstyle{remark}

%% file: timePer2Pt.bbl
\providecommand{\bysame}{\leavevmode\hbox to3em{\hrulefill}\thinspace}
\providecommand{\MR}{\relax\ifhmode\unskip\space\fi MR }
\providecommand{\MRhref}[2]{%
  \href{http://www.ams.org/mathscinet-getitem?mr=#1}{#2}
}
\providecommand{\href}[2]{#2}
\begin{thebibliography}{10}

\bibitem{DTV2014a}
B.~Deconinck, T.~Trogdon, and V.~Vasan, \emph{The method of {F}okas for solving
  linear partial differential equations}, SIAM Rev. \textbf{56} (2014), no.~1,
  159--186.

\bibitem{Duj2009a}
G.~M. Dujardin, \emph{Asymptotics of linear initial boundary value problems
  with periodic boundary data on the half-line and finite intervals}, Proc. R.
  Soc. Lond. Ser. A Math. Phys. Eng. Sci. \textbf{465} (2009), no.~2111,
  3341--3360.

\bibitem{Erd1956a}
A.~Erd\'elyi, \emph{Asymptotic expansions}, Dover, New York, 1956.

\bibitem{Fok1997a}
A.~S. Fokas, \emph{A unified transform method for solving linear and certain
  nonlinear {PDE}s}, Proc. R. Soc. Lond. Ser. A Math. Phys. Eng. Sci.
  \textbf{453} (1997), 1411--1443.

\bibitem{Fok2008a}
\bysame, \emph{A unified approach to boundary value problems}, CBMS-SIAM, 2008.

\bibitem{fvdw2021}
A.~S. Fokas and M.C.~Van der Weele, \emph{The unified transform for evolution
  equations on the half-line with time-periodic boundary conditions}, preprint
  (2021).

\bibitem{FL2012b}
A.~S. Fokas and J.~Lenells, \emph{The unified method: {I}{I} {N}{L}{S} on the
  half-line with t-periodic boundary conditions}, J. Phys. A \textbf{45}
  (2012), 195202.

\bibitem{FP2005a}
A.~S. Fokas and B.~Pelloni, \emph{A transform method for linear evolution
  {PDE}s on a finite interval}, IMA J. Appl. Math. \textbf{70} (2005),
  564--587.

\bibitem{Hop1919a}
J.~W. Hopkins, \emph{Some convergent developments associated with irregular
  boundary conditions}, Trans. Amer. Math. Soc. \textbf{20} (1919), 245--259.

\bibitem{Jac1915a}
D.~Jackson, \emph{Expansion problems with irregular boundary conditions}, Proc.
  Amer. Acad. Arts Sci. \textbf{51} (1915), no.~7, 383--417.

\bibitem{Lan1929a}
R.~E. Langer, \emph{The asymptotic location of the roots of a certain
  transcendental equation}, Trans. Amer. Math. Soc. \textbf{31} (1929), no.~4,
  837--844.

\bibitem{Lan1931a}
\bysame, \emph{The zeros of exponential sums and integrals}, Bull. Amer. Math.
  Soc. \textbf{37} (1931), 213--239.

\bibitem{Loc2000a}
J.~Locker, \emph{Spectral theory of non-self-adjoint two-point differential
  operators}, Mathematical Surveys and Monographs, vol.~73, American
  Mathematical Society, Providence, Rhode Island, 2000.

\bibitem{Pel2005a}
B.~Pelloni, \emph{The spectral representation of two-point boundary-value
  problems for third-order linear evolution partial differential equations},
  Proc. R. Soc. Lond. Ser. A Math. Phys. Eng. Sci. \textbf{461} (2005),
  2965--2984.

\bibitem{Smi2012a}
D.~A. Smith, \emph{Well-posed two-point initial-boundary value problems with
  arbitrary boundary conditions}, Math. Proc. Cambridge Philos. Soc.
  \textbf{152} (2012), 473--496.

\bibitem{Smi2012b}
\bysame, \emph{Well-posedness and conditioning of 3rd and higher order
  two-point initial-boundary value problems}, arXiv:1212.5466 [math.AP], 2012.

\end{thebibliography}
